\documentclass[11pt,reqno]{amsart}
\usepackage{amsmath,epsfig,graphicx,color}
\usepackage{ifthen}
\usepackage{amssymb,latexsym}
\usepackage{subfigure}
\usepackage{hyperref}
\hypersetup{
urlcolor=black, 
  menucolor=black, 
  citecolor=black, 
  anchorcolor=black, 
  filecolor=black, 
  linkcolor=black, 
  colorlinks=true,
}
\usepackage{comment}
\usepackage[numbers,sort&compress]{natbib} %for citation order in lists
\newcommand{\zero}{\overline{\mathbf{0}}}
\renewcommand{\H}{{\mathbf H}}
\newcommand{\wt }{\widetilde}
\newcommand{\Zt }{\widetilde{Z}^2}
\newcommand{\ov}{\overline}

\newcommand{\ld}{large deviation}
\newcommand{\bfw}{{\mathbf w}}
\newcommand{\bfc}{{\mathbf c}}
\newcommand{\bfv}{{\mathbf v}}
\newcommand{\bfA}{{\mathbf A}}

\newcommand{\bfU}{{\mathbf U}}

\newcommand{\bfD}{{\mathbf D}}

\newcommand{\bfN}{{\mathbf N}}
\newcommand{\bfK}{{\mathbf K}}
\newcommand{\bfM}{{\mathbf M}}

\newcommand{\bfH}{{\mathbf H}}
\newcommand{\bfP}{{\mathbf P}}

\newcommand{\bfI}{{\mathbf I}}
\newcommand{\bfX}{{\mathbf X}}
\newcommand{\bfu}{{\mathbf u}}
\newcommand{\bfQr}{{\mathbf Q}}
\newcommand{\bfQrm}{{\mathbf Q}^{(m)}}

\newcommand{\cmt}{continuous mapping theorem}
\newcommand{\fct}{function}
\newcommand{\slvary}{slowly varying}
\newcommand{\regvar}{regular variation}
\newcommand{\regvary}{regularly varying}
\newcommand{\st}{such that}

\newcommand{\std}{\stackrel{\rm d}{\rightarrow}}
\newcommand{\stp}{\stackrel{\P}{\rightarrow}}
\newcommand{\la}{\lambda}

\newcommand{\ds}{distribution}
\newcommand{\beao}{\begin{eqnarray*}}
\newcommand{\eeao}{\end{eqnarray*}}
\newcommand{\beam}{\begin{eqnarray}}
\newcommand{\eeam}{\end{eqnarray}}
\newcommand{\barr}{\begin{array}}
\newcommand{\earr}{\end{array}}
\newcommand{\red}{\color{darkred}}
\newcommand{\blue}{\color{darkblue}}

\definecolor{darkblue}{rgb}{.1, 0.1,.8}
\definecolor{darkgreen}{rgb}{0,0.8,0.2}
\definecolor{darkred}{rgb}{.8, .1,.1}
\newcommand{\bco}{\begin{corrolary}}
\newcommand{\eco}{\end{corrolary}}

\newcommand{\E}{\mathbb{E}}
\renewcommand{\P}{\mathbb{P}}
\newcommand{\1}{\mathbf{1}}
\newcommand{\R}{\mathbb{R}}
\newcommand{\N}{\mathbb{N}}
\newcommand{\C}{\mathbb{C}}
\newcommand{\bfC}{{\mathbf C}}

\newcommand{\Z}{\mathbb{Z}}

\DeclareMathOperator{\e}{e}
\DeclareMathOperator{\Proj}{Proj}

\newcommand{\x}{{\mathbf x}}
\newcommand{\y}{{\mathbf y}}
\newcommand{\X}{{\mathbf X}}

\newcommand{\M}{{\mathbf M}}

\newcommand{\z}{{\mathbf Z}}

\newcommand{\dint}{\,\mathrm{d}}

\newcommand{\twonorm}[1]{\|#1\|_2}

\newcommand{\frobnorm}[1]{\|#1\|_F}
\newcommand{\ltwonorm}[1]{\|#1\|_{\ell_2}}
\newcommand{\vep}{\varepsilon}
\newcommand{\nto}{n \to \infty}
\newcommand{\xto}{x \to \infty}

\newcommand{\rhs}{right-hand side}
\newcommand{\ts}{time series}
\newcommand{\tsa}{\ts\ analysis}
\newcommand{\fidi}{finite-dimensional distribution}
\newcommand{\rv}{random variable}

\newcommand{\Dr}{D}

\newcommand{\MP}{Mar\v cenko--Pastur }

\newcommand{\ex}{{\rm e}\,}

\def\tag{\refstepcounter{equation}\leqno }

\newtheorem{lemma}{Lemma}[section]

\newtheorem{theorem}[lemma]{Theorem}
\newtheorem{proposition}[lemma]{Proposition}

\newtheorem{example}[lemma]{Example}

\newtheorem{remark}[lemma]{Remark}

\newcommand{\cid}{\stackrel{d}{\rightarrow}}
\newcommand{\cip}{\stackrel{\P}{\rightarrow}}

\newcommand{\as}{{\rm a.s.}}

\newcommand{\pp}{point process}
\newcommand{\con}{convergence}
\newcommand{\seq}{sequence}
\newcommand{\ms}{measure}
\newcommand{\asy}{asymptotic}
\begin{document}
\today
\bibliographystyle{acm}
\title[Large sample autocovariance matrices for linear processes]
{Large sample autocovariance matrices  of linear processes with heavy tails}
\thanks{Johannes Heiny was supported by the Deutsche Forschungsgemeinschaft (DFG) via RTG 2131 High-dimensional Phenomena in Probability -- Fluctuations and Discontinuity. Thomas Mikosch's research was
generously supported by an Alexander von Humboldt Research Award.}
\author[Johannes Heiny]{Johannes Heiny}
\author[Thomas Mikosch]{Thomas Mikosch}
\address{Fakult\"at f\"ur Mathematik,
Ruhruniversit\"at Bochum,
Universit\"atsstrasse 150,
D-44801 Bochum,
Germany}
\email{johannes.heiny@rub.de}
\address{Department  of Mathematical Sciences,
University of Copenhagen,
Universitetsparken 5,
DK-2100 Copenhagen,
Denmark}
\email{mikosch@math.ku.dk\,, www.math.ku.dk/$\sim$mikosch}
\begin{abstract}
We provide asymptotic theory for certain functions of the sample autocovariance 
matrices of a high-dimensional time series with infinite fourth moment. The time series exhibits linear dependence across the coordinates and through time.
Assuming that the dimension increases with the sample size, we provide theory for the eigenvectors of the sample autocovariance matrices and find explicit approximations of a simple structure, whose finite sample quality is illustrated for simulated data. 
We also obtain the limits of the normalized eigenvalues of functions of the sample autocovariance matrices in terms of cluster Poisson point processes. In turn, we derive the distributional limits of the largest eigenvalues and functionals acting on them. In our proofs, we use 
large deviation techniques for heavy-tailed processes, point process techniques motivated by extreme value theory, and related continuous mapping arguments.
\end{abstract}
\keywords{Regular variation, sample autocovariance matrix, linearly dependent entries,
largest  eigenvalues, trace, point process
  convergence, cluster Poisson limit,
infinite variance stable limit, Fr\'echet distribution, large deviations}
\subjclass{Primary 60B20; Secondary 60F05 60F10 60G10 60G55 60G70}

\maketitle
%\tableofcontents

%---------------------------------------------------------------------------
\section{Introduction}\label{sec:intro}\setcounter{equation}{0}
\subsection{Some history}
In \tsa\ 
the notions of autocovariance, autocorrelation and their sample versions
are basic tools for the study of the (linear) dependence 
structure, spectral analysis, 
parameter estimation, goodness-of-fit, change-point detection, etc.; see 
for example the classical monographs \cite{brockwell:davis:1991,priestley:1981}. When considering random matrices $\X=\X_n= (\x_1,\ldots, \x_n)$
with high-dimensional time series observations $\x_t= (X_{1t},\ldots,X_{pt})'$, $t\in \Z$, the main focus of interest 
has been on the limiting spectral \ds\ of $\bfX$ and on the \asy\
properties of the eigenvalues and eigenvectors of the {\em sample covariance matrix} $\X\X'$;
%\beao
%\X\X' = \big(\sum_{t=1}^n X_{it}X_{jt}\big)_{i,j=1,\ldots,p}\,;
%\eeao
see for instance \cite{bai:silverstein:2010}.
From the observations $(\x_t)_{t\in \Z}$ one can also construct the $p\times n$ matrices
\begin{equation}\label{eq:26}
\X_n(s)= (\x_{1+s},\ldots, \x_{n+s}) \,,\qquad s=0,1,2,\ldots\,,
\end{equation}
while we refer to $\X=\X_n(0)$ as the {\em data matrix}.
Now, in analogy with the sample autocovariance \fct\ of a stationary process,  we introduce the (non-normalized)
{\em sample autocovariance matrices at lag $s$}:
\beam\label{eq:sample1}
\X_n(0)\X_n(s)'\,,\qquad s=0,1,2,\ldots\,.
\eeam
For $s=0$, we obtain the sample covariance matrix.
\par
To the best of our knowledge,
the idea of using functions of the {\em sample autocovariance matrices} 
originates from the  paper \cite{lam:yao:2012}. 
The authors work in the framework
of factor models for $\bfX$ and under light-tail assumptions on the 
entries $X_{it}$. 
The main goal of using  sample autocovariance matrices  
in \cite{lam:yao:2012} was to
derive a rule for determining a number of significant 
eigenvalues and eigenvectors for principal component analysis (PCA) in a high-dimensional \ts\ setting.
This was achieved by exploiting the additional information about the
dependence of the \ts\ $(\x_t)_{t\in \Z}$, contained
in the sample autocovariance matrices  $\X_n(0)\X_n(s)'$ for different lags $s\ge 0$.
\par
Recently, a whole series of articles on sample autocovariance matrices
was published. Again, factor models are assumed for 
describing the dynamics of the multivariate \ts\  
$(\x_t)_{t\in \Z}$. 
The authors of \cite{li:wang:yao:2017} study
a ratio estimator for the number of relevant eigenvalues 
based on singular values of lagged
sample autocovariance matrices.
The paper proposes a complete theory of such sample singular
values for both the factor and noise parts under the large-dimensional
scheme where the dimension and the sample size grow proportionally to infinity. The papers \cite{wang:yao:2015,wang:yao:2016} consider a 
moment approach for determining the limiting spectral \ds\ of the 
singular values of the autocovariance matrices and for deriving the 
\con\ of the largest singular value.
The limiting spectral \ds\ of a symmetrized 
sample autocovariance matrix is studied in  
\cite{jin:wang:chen:bai:nair:2014,bai:wang:2015, liu:aue:paul:2015, wang:aue:paul:2017} while
\cite{wang:jin:bai:nair:2015} consider the extreme eigenvalues
of such a matrix. The limiting  spectral distribution 
of sample autocovariance matrices for factor models
is investigated in \cite{li:pan:yao:2015}.

\subsection{Our model}
In this paper we study the singular values of \fct s of the 
sample autocovariance matrices \eqref{eq:sample1} at different lags $s$.
Our model assumptions are quite distinct from most of the literature.
\subsubsection*{Growth condition on $p$}
We describe high-dimensionality of the time series observations $\x_t= (X_{1t},\ldots,X_{pt})'$
by assuming that $p=p_n\to \infty$ as $\nto$.
To be precise, we assume an integer sequence
\begin{equation}\label{eq:p}
p=p_n=n^\beta \ell(n), \quad n\ge1,\tag{$C_p(\beta)$}
\end{equation}
where $\ell$ is a slowly varying function and $\beta\in [0,1]$. 
This condition is more general than the growth conditions in the literature; 
see for example \cite{johnstone:2001,auffinger:arous:peche:2009,tao:vu:2012}, 
where it is assumed that $p/n\to\gamma\in (0,\infty)$. Condition
\eqref{eq:p} is also more general than in
\cite{davis:mikosch:pfaffel:2016,davis:pfaffel:stelzer:2014} 
who have restrictions on the size of $\beta$, depending
on the heaviness of the tails of $\X$.
\subsubsection*{Linear dependence}
From a time series perspective it is natural to assume dependence
between the entries $X_{it}$ both through time $t$ and across the rows $i$. 
In the aforementioned literature, dependence through time and across rows
is often described by a factor model. This kind of model has been successfully
used in econometrics. 
\par
We assume  {\em linear dependence}
between the rows and columns of $\X$:
\begin{equation}\label{eq:1}
X_{it}=\sum_{l\in \Z}\sum_{k\in \Z} h_{kl} \,Z_{i-k,t-l}\,,\qquad i,t\in\Z\,,
\end{equation}
where $(Z_{it})_{i,t\in \Z}$ is an iid  field of \rv s 
with generic element $Z$ and  $(h_{kl})_{k,l\in\Z}$ 
is a suitable array of real numbers \st\ the infinite series in  \eqref{eq:1}
converges a.s.; see \eqref{eq:2a} below. 
\par
The model \eqref{eq:1} was introduced in \cite{davis:pfaffel:stelzer:2014}, assuming the rows iid, and, in the
present form, used in \cite{davis:mikosch:pfaffel:2016,davis:heiny:mikosch:xie:2016}. Linear dependence is a natural concept in \tsa ; it
also allows one to describe the asymptotic properties of 
the eigenstructure of $\X\X'$ in  a transparent way.
\subsubsection*{Heavy-tail condition} 
In all the existing literature on sample autocovariance matrices
it is assumed that the 4th moment of the entries $X_{it}$ is finite.
We will refer to this condition as {\em light tails}. If 4th moments are
infinite we instead refer to {\em heavy tails.} The reason for this
distinction is that there is a phase transition in the limit behavior
of the largest eigenvalues of the sample covariance matrix and, as we will see later, also of the largest singular values of the  sample autocovariance matrices.  
\par
In the case of 
iid light-tailed $(X_{it})$ it is known that the 
largest eigenvalue of $\bfX\bfX'$ typically has a Tracy-Widom 
limit \ds ; see for example \cite{johnstone:2001,tao:vu:2012}
for benchmark results. This is in sharp contrast to the heavy-tail case.
Due to work by \cite{soshnikov:2004,soshnikov:2006,auffinger:arous:peche:2009} we know that the 
largest eigenvalue of the suitably normalized matrix $\bfX\bfX'$
has a Fr\'echet limit \ds , 
\beao
\Phi_{\alpha/2}(x)=\ex^{-x^{-\alpha/2}}\quad \mbox{for some $\alpha\in (0,4)$,}
\eeao
which is one of the max-stable \ds s,
i.e., one of the limit \ds s of normalized and centered maxima
of an iid \seq , 
see \cite[Chapter~3]{embrechts:kluppelberg:mikosch:1997}.
\par   
The assumption of infinite 4th moment is not sufficient to derive a precise
weak limit theory for eigenvalues and singular values. Therefore,
as in 
\cite{soshnikov:2006,auffinger:arous:peche:2009,heiny:mikosch:2017:iid} in the iid case and in
 \cite{davis:pfaffel:stelzer:2014,davis:mikosch:pfaffel:2016,davis:heiny:mikosch:xie:2016} in the linear dependence case 
we assume that $Z$ is {\em \regvary } in the sense that the following tail balance condition holds
\begin{equation}\label{eq:27}
\P(Z>x)\sim p_+ \dfrac{L(x)}{x^{\alpha}}\quad\mbox{and}\quad  \P(Z<-x)\sim p_-
\dfrac{L(x)}{x^{\alpha}}\,,\qquad \xto\,,
\end{equation}
for some tail index $\alpha\in (0,4)$, constants $p_+,p_-\ge 0$ with $p_++p_-=1$ and a \slvary\ \fct\ $L$. The \regvar\ condition for $\alpha\in (0,4)$ 
implies that we consider the heavy-tail case where 
both $\E[Z^4]=\infty$ and $\E[X^4]=\infty$; see \cite{davis:heiny:mikosch:xie:2016, heiny:mikosch:2017:iid, heiny:mikosch:2017:corr} for collections of results which show  
the stark differences between the heavy-tail and light-tail cases.
In addition, we assume $\E[ Z]=0$ 
whenever $\E [|Z|]<\infty$.
Moreover, we also require the {\em summability condition}
\begin{equation}\label{eq:2a}
\sum_{l \in \Z} \sum_{k\in \Z} |h_{kl}|^{\delta} <\infty\quad\mbox{for some $\delta\in (0,\min({\alpha/2},1))$.}
\end{equation} 
The conditions \eqref{eq:27}, \eqref{eq:2a} and $\E[Z]=0$ if $\E[|Z|]<\infty$ ensure
the a.s.~absolute convergence of the series in \eqref{eq:1}. Moreover, the marginal and
\fidi s of the field $(X_{it})$ are \regvary\ with index $\alpha$; see
for example \cite{embrechts:kluppelberg:mikosch:1997}, Appendix A3.3. Therefore we also refer to $(X_{it})$ and $(Z_{it})$
as {\em \regvary\ fields.} Notice that \regvar\ of $(Z_{it})$ and the \con\ of \eqref{eq:1} imply that $(X_{it})$ constitutes a strictly stationary random field; we denote a generic element by $X$.

\subsection{Functions of sample autocovariance matrices}
Recall the definition of the {\em sample autocovariance matrix at lag $s$} from 
\eqref{eq:sample1}.
%From the field $(X_{it})$ we can construct the $p\times n$ matrices
%\begin{equation}\label{eq:26}
%\X_n(s)= (X_{i,t+s})_{i=1,\ldots,p;t=1,\ldots,n}\,,\qquad s=0,1,2,\ldots\,,
%\end{equation}
%while we often use the previous notation $\X=\X_n(0)$.
%Now we can introduce the (non-normalized)
%\beam\label{eq:sample1}
%\X_n(0)\X_n(s)'\,,\qquad s=0,1,2,\ldots\,.
%\eeam
%For $s=0$, we obtain the {\em sample covariance matrix.}
%To the best of our knowledge,
%the idea of using functions of the sample autocovariance matrices originates 
%from a paper by Lam and Yao~\cite{lam:yao:2012} in the context of factor models which are distinct
%from \eqref{eq:1}, and their assumptions on $(X_{it})$ imply light tails. The main goal in \cite{lam:yao:2012} was to
%derive a rule for determining a number of significant eigenvalues and eigenvectors for principal component analysis (PCA) 
%in a high-dimensional \ts\ $\X$.
We are  interested in the asymptotic behavior (of \fct s) of the eigen- and singular values of the matrices
\begin{eqnarray}\label{eq:sample}
\bfC_n(s) =
\left\{\begin{array}{ll}
\X_n(0)\X_n(s)' \,, & \mbox{if } \alpha<2(1+\beta), \\
 \X_n(0)\X_n(s)'-\E[\X_n(0)\X_n(s)'] \,, & \mbox{if } \alpha>2(1+\beta),
\end{array}\right. \qquad s=0,1,2,\ldots .
\end{eqnarray}\noindent
For $\alpha>2(1+\beta)$, the centering $\E[\X_n(0)\X_n(s)']$ is needed to ensure a non-degenerate limiting spectrum of $\bfC_n(s)$. A similar centering was used in \cite{davis:pfaffel:stelzer:2014,davis:mikosch:pfaffel:2016,davis:heiny:mikosch:xie:2016}.
The case $\alpha= 2(1+\beta)$ is slightly more technical, but can be handled as well; see Remark \ref{rem:speciala} below. 
\par
The eigenvalues of the non-symmetric matrix $\bfC_n(s)$ for $s\ge 1$ can be complex. One way to avoid this is
to calculate the {\em singular values} of this matrix,
i.e., the square roots of the eigenvalues of the non-negative definite matrix 
$\bfC_n(s) \bfC_n(s)'$. The largest of these singular
values is the {\em spectral norm} $\twonorm{\bfC_n(s)}$.

In this paper, we study the \asy\ behavior of the eigenvalues and eigenvectors of the sum
\begin{equation}\label{eq:Ps1s2}
\bfP_n(s_1,s_2)= \sum_{s=s_1}^{s_2} \bfC_n(s) \bfC_n(s)'\quad \mbox{for fixed $0\le s_1\le s_2$.}
\end{equation}
In what follows, we will often suppress the dependence of $\bfC_n$ and $\bfP_n$ on $n$ and simply 
write $\bfC$ and $\bfP$.
This research is motivated by \cite{lam:yao:2012} who considered the ratio of successive largest
eigenvalues of $\bfP_n(1,s)$ for various values $s\ge 1$. The goal was to find a value $s$
\st\ the relevant information about the eigenvalues contained in the sample autocovariances
$\X_n(0)\X_n(s)'$ is exhausted.

\subsection{Motivation and structure of this paper}
When looking at the coordinates of the eigenvectors of sample autocovariance matrices of 
financial time series, we noticed that certain patterns, in particular around the 
largest coordinate values, occurred repetitively in several eigenvectors. 
Our goal was to find some theoretical explanation for this phenomenon. 
Another challenge was added by the stylized fact
that financial time series are heavy-tailed. In contrast, 
most of the literature on dimension reduction and high-dimensional time series focuses 
on the light-tailed case. 
\par
We assume a linear dependence structure through time and across the rows for the underlying \ts . The 
eigenvalues of large sample covariance matrices of linear processes was already studied  
in Davis et al. \cite{davis:pfaffel:stelzer:2014,davis:heiny:mikosch:xie:2016,davis:mikosch:pfaffel:2016}.
The sample autocovariance matrices call for additional challenges since they require to understand
the interplay between the largest values of the noise $(Z_{it})$, the lag $s$ and the coefficient matrix $(h_{kl})$.
We use \ld\ theory for sums of heavy-tailed \rv s in combination with \pp\ \con\ results and continuous
mapping arguments to derive asymptotic theory for the eigenvectors and eigenvalues of 
large sample autocovariance matrices for time series with infinite fourth moment. Our results
are very explicit as regards the dependence structure and magnitude of the largest eigenvalues
as well as the construction of the corresponding eigenvectors.
\par
This paper is organized as follows. Due to the complexity of the model the notation in this paper is rather involved. Therefore, 
in Section~\ref{sec:2}, we introduce the most important quantities used throughout the paper.
In Section~\ref{sec:3}, we present the main \asy\ results. Theorem~\ref{thm:mainn} provides
explicit approximations to the eigenvalues of the matrix sum  $\bfP_n(s_1,s_2)$. The major contribution of this work is the description of the  eigenvectors of $\bfP_n(s_1,s_2)$.  Theorem~\ref{thm:eigenvector} contains explicit approximations of these eigenvectors under the additional restriction that 
the coefficient matrix $(h_{kl})$ has only finitely many non-zero entries. Extensions to
coefficient matrices with infinitely many non-zero entries seem possible under an additional condition on the decay of $h_{kl}$ for $|k|\vee |l|\to \infty$. In Section~\ref{subsec:32} we also
include detailed examples of our proposed eigenvector and eigenvalue approximations for simulated data and the S\&P 500 log-returns. In  Figures \ref{fig:abc},\ref{fig:00} and \ref{fig:11} the reader can convince him/herself with the naked eye that the eigenvectors possess the structure predicted by our asymptotic theory.
Theorem~\ref{thm:pp}
presents results on the weak \con\ of the \pp\ of the normalized eigenvalues of $\bfP_n(s_1,s_2)$
towards some cluster Poisson process. The limiting \pp\ allows one to derive the \asy\
structure of the largest eigenvalues of  $\bfP_n(s_1,s_2)$. Applications of the 
\cmt\ yield \asy\ theory for \fct als acting on the \seq\ of the eigenvalues
such as the spectral gap, the ratio of the largest eigenvalue and the trace. In 
Section~\ref{subsec:34} we derive analogous results on the eigenstructure 
of sums of the  symmetrized matrices $(\bfC_n(s)+\bfC_n(s)')/2$. Section~\ref{sec:limit} describes the limiting spectral distribution of the sample covariance matrix $\X\X'$ when $p$ and $n$ are proportional. Section~\ref{sec:proof} contains the proofs 
of the main theorems.

%----------------------------------------------------------------------------
\section{More notation}\label{sec:2}\setcounter{equation}{0}
Before we can formulate the main results we have to introduce relevant notation to be used throughout.
\subsubsection*{Order statistics}
The order statistics of the field
$(Z_{it}^2)_{i=1,\ldots,p;t=1,\ldots,n}$ 
 \begin{equation}\label{eq:zorder}
Z_{(1),np}^2 \ge Z_{(2),np}^2 \ge  \cdots \ge Z_{(np),np}^2, \qquad n,p\ge 1\,.
\end{equation}
\subsubsection*{Sums of squares}

\begin{eqnarray}\label{eq:newD}
D_i=D_i^{(n)} =
\left\{\begin{array}{ll}
\sum_{t=1}^n Z_{it}^2 \,, & \mbox{if } \alpha<2(1+\beta), \\
\sum_{t=1}^n Z_{it}^2 - n\E[Z^2] \,, & \mbox{if } \alpha>2(1+\beta),
\end{array}\right. \qquad i=1,\ldots,p; \quad n\ge 1\,,
\end{eqnarray}\noindent
with generic element $D$ and their ordered squared values for fixed $n$,
\beam\label{eq:help6}
D_{(1)}^2=D_{L_1}^2\ge \cdots \ge D_{(p)}^2=D_{L_p}^2\,.
\eeam
We assume without loss of generality that $(L_1,\ldots,L_p)$ is
a permutation of $(1,\ldots,p)$.

\subsubsection*{The matrices $\M(s)$ and $\bfK(s_1,s_2)$} We introduce some auxiliary matrices
derived from the coefficients $(h_{kl})_{k,l\in \Z}$:
\beao
\H(s)=(h_{k,l+s})_{k,l\in \Z}, \qquad \M(s)= \H(0)\H(s)'\,,\qquad s\ge 0\,.
\eeao
Notice that
\begin{equation}\label{eq:m}
(\M(s))_{ij}= \sum_{l\in \Z} h_{i,l} h_{j,l+s}, \qquad i,j \in \Z .
\end{equation}
For $0\le s_1\le s_2 <\infty$, we define the positive semi-definite matrix
\begin{equation}\label{eq:Ks1s2}
\bfK(s_1,s_2)= \sum_{s=s_1}^{s_2} \M(s)\M(s)'
\end{equation}
and denote its ordered eigenvalues by 
\begin{equation}\label{eq:v1}
v_1^2(s_1,s_2) \ge v_2^2(s_1,s_2) \ge\cdots \,.
\end{equation}
We interpret $v_i(s_1,s_2)$ as the positive square root of $v_i^2(s_1,s_2)$.

Throughout this paper we assume that $\bfK(s_1,s_2)$ is not the null-matrix. Let $r(s_1,s_2)$ be the rank of $\bfK(s_1,s_2)$ so that $v_{r(s_1,s_2)}(s_1,s_2)>0$ while $v_{r(s_1,s_2)+1}(s_1,s_2)=0$ if $r(s_1,s_2)$ is
finite, otherwise $v_i(s_1,s_2)>0$ for all $i$.
\par
The singular values of $ \M(s)$ are $(v_i(s,s))$; for ease of notation we will 
sometimes denote them by $ (v_i(s))$.
Under the summability condition \eqref{eq:2a} on $(h_{kl})$ for fixed $0\le s_1\le s_2<\infty$, denoting the Frobenius norm by~$\|\cdot\|_F$,
\beam\label{eq:tracea}
\sum_{i=1}^\infty v_i^2(s_1,s_2) &= & \sum_{s=s_1}^{s_2} \sum_{i=1}^\infty v_i^2(s)=
  \sum_{s=s_1}^{s_2} \frobnorm{\M(s)}^2= \sum_{s=s_1}^{s_2} \sum_{i,j\in \Z} \sum_{l_1,l_2 \in \Z} h_{i,l_1} h_{j,l_1+s}h_{i,l_2}
h_{j,l_2+s}\nonumber\\
&\le& { \sum_{s=s_1}^{s_2} \Big(\sum_{l_1,l_2 \in \Z} \sum_{i\in \Z} |h_{i,l_1} h_{i,l_2}|\Big)^2}
\le c \, (s_2-s_1) \sum_{l_1\in \Z} \sum_{i \in \Z} |h_{i,l_1}| <\infty\,.
\eeam
Therefore all eigenvalues $v_i^2(s_1,s_2)$ are finite and the ordering
  \eqref{eq:v1}
is justified.\\[1mm]
{\em Here and in what follows, we write $c$ for any positive constant whose value is not of interest.}

\subsubsection*{Normalizing sequence}
We define $(a_k)$ by
\beao
\P(|Z|>a_k)\sim k^{-1}\,,\qquad k\to\infty\,.
\eeao
The largest eigenvalues of random matrices will typically be 
normalized by the \seq\ $(a_{np}^4)$, where $p$ is given in \eqref{eq:p}.
\subsubsection*{Eigenvalues of the sample autocovariance matrices} The ordered eigenvalues of 
$\bfP(s_1,s_2)$  in \eqref{eq:Ps1s2} are
\beam\label{eq:la}
\lambda_{1}(s_1,s_2) \ge \cdots \ge \lambda_{p}(s_1,s_2)\,,
\eeam
where we suppress the dependence on $n$ in the notation. 

\subsubsection*{Approximations to eigenvalues}\label{subsec:defdelta}
Approximations to the ordered eigenvalues $\lambda_{i}(s_1,s_2)$  will be given in
terms of the ordered values 
\beam\label{eq:gammadelta}\barr{ll}
\delta_{1}(s_1,s_2)\ge \cdots \ge \delta_{p}(s_1,s_2)&\mbox{from}\quad \big\{Z_{(i),np}^4 v_j^2(s_1,s_2)\,, i=1,\ldots,p\,;j=1,2,\ldots\big\}\,,\\
\gamma_{1}(s_1,s_2)\ge \cdots \ge \gamma_{p}(s_1,s_2)&\mbox{from}\quad
\big\{D_{i}^2 v_j^2(s_1,s_2), i=1,\ldots,p\,;j=1,2,\ldots\big\}\,.
\earr
\eeam

%---------------------------------------------------------------------------
\section{Approximations of eigenvalues and eigenvectors}\label{sec:3}\setcounter{equation}{0}
%---------------------------------------------------------------------------
In this section we provide the main approximation results for the ordered eigenvalues and the corresponding  
eigenvectors of the
sample autocovariance matrices of the linear model \eqref{eq:1}. The relevant notation is given 
in Sections~\ref{sec:intro} and \ref{sec:2}. 

\subsection{Eigenvalues of the sample autocovariance \fct}\label{sec:mainresult}

\begin{theorem}[Eigenvalues of $\bfP(s_1,s_2)$]\label{thm:mainn}
Consider the linear process \eqref{eq:1} under
\begin{itemize}
\item
the growth condition \eqref{eq:p} on $(p_n)$ for some $\beta\in[0,1]$,
\item
the \regvar\ condition \eqref{eq:27}
for some $\alpha\in (0,4)\backslash \{2(1+\beta)\}$,
\item the centering condition
$\E[Z]=0$ if $\E[|Z|]$ is  finite,
\item
the summability condition
\eqref{eq:2a} on the coefficient matrix  $(h_{kl})$.
\end{itemize}
Then we have for $0\le s_1\le s_2<\infty$,
\begin{equation}\label{eq:mainn1}
a_{np}^{-4} \max_{i=1,\ldots,p} |\la_{i}(s_1,s_2)-\gamma_{i}(s_1,s_2)| \cip 0, \quad \nto.
\end{equation}
Moreover, if $\alpha<2(1+\beta)$, then
\begin{equation}\label{eq:mainn2}
a_{np}^{-4} \max_{i=1,\ldots,p} |\la_{i}(s_1,s_2)-\delta_{i}(s_1,s_2)| \cip 0, \quad \nto.
\end{equation}
\end{theorem}
The technical and quite lengthy proof of Theorem \ref{thm:mainn} can be found in Sections~\ref{sec:proof1}-\ref{subsec:Thm6.2}.
 %Section~\ref{subsec:proofthm:mainn}. 
For the convenience of the reader, we first present the main ideas of the proof in the setting of a filter $(h_{kl})$ with finitely many non-zero entries in Section~\ref{subsec:sketchproofthm:mainn}.
\par
The approximations \eqref{eq:mainn1} and \eqref{eq:mainn2} are strikingly simple considering the high dimension of $\bfP(s_1,s_2)$:
apart from multiplication with the deterministic $(v_j^2(s_1,s_2))$, the approximating values in Theorem \ref{thm:mainn} are just the order statistics of the iid sequences $(Z^4_{it})_{i\le p; t\le n}$ and $(D_i^2)_{i\le p}$, respectively.

\begin{example}\label{ex:eigenvalues}{\em
We analyze a specific structure of the coefficients $h_{kl}$ of the linear process \eqref{eq:1} and consider the {\em separable case}, i.e.,
\beao
h_{kl}=\left\{\barr{ll}
d_k\,c_l\,,& k,l \ge 0\,,\\
0& \mbox{otherwise,}\earr
\right.
\eeao
for given real \seq s
$
{\mathbf d} = (d_0, d_1,d_2,\ldots)'$ and ${\mathbf c}= (c_0,c_1,c_2, \ldots)'$,
where we assume $d_0>0$ and that ${\mathbf c}$ is not the null sequence.

First, we determine the values $(\delta_i(s_1,s_2))$ and $(\gamma_i(s_1,s_2))$ which approximate the eigenvalues of the autocovariance function in Theorem~\ref{thm:mainn}.
The matrix $\bfD={\mathbf d}{\mathbf d}'$ is symmetric, has rank one 
and the only non-zero eigenvalue is $\ov d=\sum_{k=0}^\infty d_k^2$. 
We conclude from \eqref{eq:m} that
\beam\label{eq:mex}
\bfM(s)=\ov c (s)\,\bfD  \,, \qquad s\ge 0\,,
\eeam
whose only non-zero eigenvalue is $\ov d\,\ov c (s)=\ov d\,\sum_{l=0}^\infty c_lc_{l+s}$.
The factors $(\ov c (s))$ can be positive or negative; they constitute the autocovariance \fct\ of a stationary linear process $Y_t=\sum_{l=0}^\infty c_l\,V_{t-l}$, $t\in{\mathbb Z}$, where $(V_i)$ is a unit variance white noise process.

From \eqref{eq:Ks1s2} and \eqref{eq:mex} we obtain for $0\le s_1\le s_2 <\infty$ that 
\begin{equation}\label{eq:Kex}
\bfK(s_1,s_2)= \sum_{s=s_1}^{s_2} \ov c ^2(s) \bfD \bfD'\,.
\end{equation}
This matrix has rank $1$ and its largest eigenvalue is given by $v_1^2(s_1,s_2) =\sum_{s=s_1}^{s_2} \ov c ^2(s) \, \, \ov d^2$.
The approximating values in \eqref{eq:mainn1} and \eqref{eq:mainn2} are therefore 
\begin{equation*}
\gamma_i(s_1,s_2)=D_{(i)}^2 \sum_{s=s_1}^{s_2} \ov c ^2(s) \, \, \ov d^2\quad \text{ and } \quad \delta_i(s_1,s_2)=Z_{(i),np}^4 \sum_{s=s_1}^{s_2} \ov c ^2(s) \, \, \ov d^2\,,\quad 1\le i\le p\,.
\end{equation*}
Moreover, we have the remarkable identity 
\begin{equation}\label{eq:ident}
\gamma_i(s_1,s_2)=\sum_{s=s_1}^{s_2}\gamma_i(s,s)
\end{equation}
which implies $\la_i(s_1,s_2)\approx \sum_{s=s_1}^{s_2}\la_i(s,s)$ for large $n$. 
For illustrations of this phenomenon on real and simulated data,
see Figure~\ref{fig:LamYao} and the end of Example~\ref{ex:eigenvectors}, respectively.

\begin{figure}[htb!]
  \centering
  \subfigure[]
{
    \includegraphics[scale=0.4]{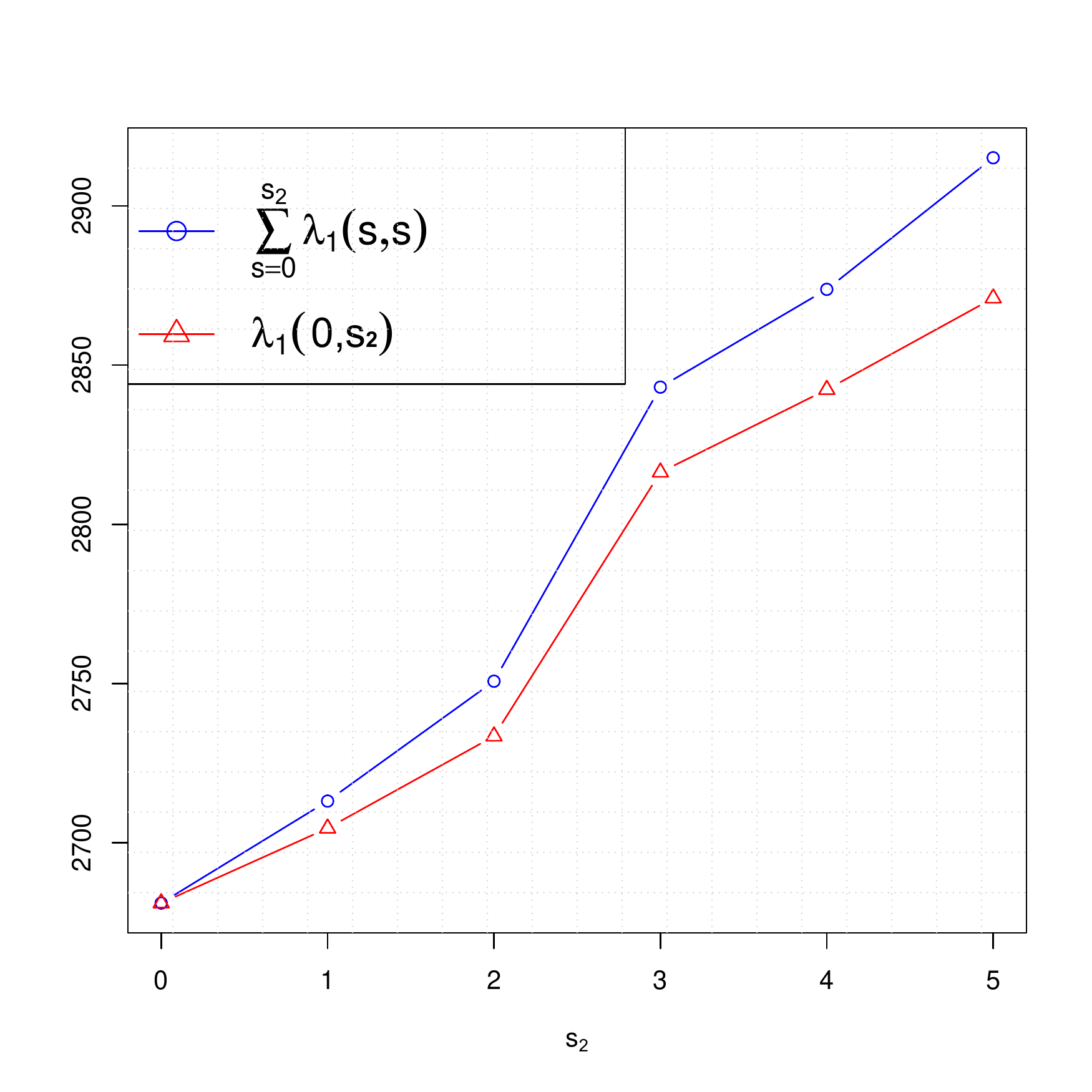}
    \label{fig:LamYao:a1}
  }
  \subfigure[] {
    \includegraphics[scale=0.38]{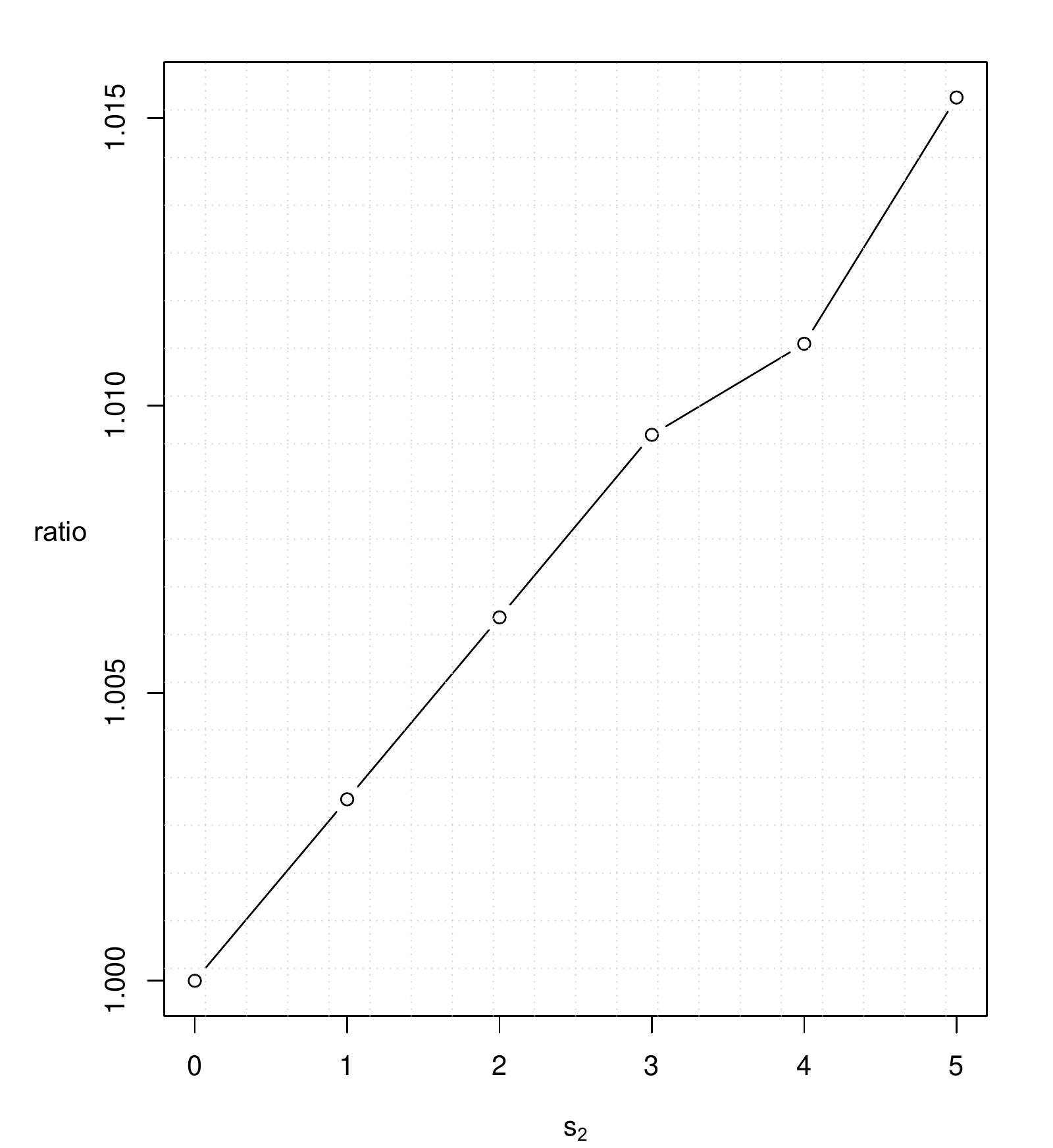}
    \label{fig:LamYao:b1}
  }
  \caption{ A comparison of the sums of the largest eigenvalues  $\lambda_1(s,s)$ of the squared autocovariance matrices and the largest eigenvalue $\lambda_1(0,s_2)$ of the sum of these matrices for $s_2=0,\ldots,5$.
The underlying data $\bfX$ consists of $p=478$  log-return series composing
the S\&P 500 index estimated from $n=1345$ daily observations from 01/04/2010 to 02/28/2015.
The two values are surprisingly close to each other; mind the scale
of the $y$-axis. We also show the corresponding ratios which are very close to one.}
  \label{fig:LamYao}
\end{figure}
}\end{example}

%-----------------------------------------------------------------------
\subsection{Eigenvectors in the linear dependence model}\label{subsec:32}
In this section, $s_1\le s_2$ are two given non-negative integers such that $\bfK(s_1,s_2)$ is not the 
null-matrix. {\em The values $(s_1,s_2)$ are of no particular interest. 
Therefore we drop them in our notation. For example, we write $\bfK$, $\bfP$, $\la_i$ 
instead of  $\bfK(s_1,s_2),\bfP(s_1,s_2),\la_i(s_1,s_2)$.}

We provide approximations of the unit eigenvectors of $\bfP$ and give explicit expressions. 
For simplicity, we assume that $(h_{kl})$ is a matrix with finitely many non-zero entries.
This means that $(X_{it})$ is a finite moving average both through time and across the rows. 

Moreover, to solve identifyability issues of eigenvectors, we require that the eigenspace belonging to each non-zero eigenvalue of the deterministic matrix $\bfK$ is one-dimensional. 
\begin{itemize}
\item {\bf Condition $H_{s_1,s_2}$:} 
(1) There exists an $m\in \N$ such that $h_{kl}=0$ if $|k|\vee |l|>m$.\\
(2) There are no ties among the non-zero eigenvalues of $\bfK$ defined in \eqref{eq:v1}.
\end{itemize}

Let $\y_{i}$ be the unit eigenvector of $\bfP$ associated with the non-zero eigenvalue $\lambda_i$, i.e., 
$\bfP \y_{i} = \lambda_i \, \y_{i}$ and $\ltwonorm{\y_{i}}=1$. Throughout we use the convention that the first non-zero component of eigenvectors is assumed positive.
Recall the definition of $\gamma_{i}$ from \eqref{eq:gammadelta} and define the random indices $a(i), b(i)$ which satisfy the equation
\begin{equation*}
\gamma_{i}= D_{a(i)}^2\,{v}_{b(i)}^2\,.
\end{equation*}

Under condition $H_{s_1,s_2}$, the matrix $\bfK$ is zero outside of a block of size $(2m+1)\times (2m+1)$. More precisely, if we set $\widehat \bfK = (\bfK_{i-m-1,j-m-1})_{i,j=1,\ldots,2m+1}$, then we have 
\begin{equation}\label{eq:Khat}
\bfK =
\begin{pmatrix}
\zero & & \\
& \widehat \bfK &\\
& & \zero
\end{pmatrix}\,,
\end{equation} 
and the non-zero eigenvalues of $\widehat \bfK$ and $\bfK$ coincide. {Here and in what follows, $\zero$ denotes a quadratic matrix consisting of zeros. 
We use this symbol to describe the structure of large matrices
where the dimension of two $\zero$'s in the same line might be distinct or random.}

By $\bfu_i=(\bfu_{i,1}, \ldots, \bfu_{i,2m+1})'$ we denote the unit eigenvector of $\widehat \bfK$ associated with ${v}_i^2>0$, $i=1,\ldots,r={\rm{rank}}(\bfK)\le 2m+1$. Under condition $H_{s_1,s_2}$, the $(\bfu_i)$ are unique.
We embed these $(2m+1)$-dimensional vectors into $p$-dimensional vectors $\bfu_i^{a}=(\bfu_{i,1}^a, \ldots,\bfu_{i,p}^a)'$, $a\in \Z$, via 
\begin{eqnarray}\label{eq:matrixu}
\bfu_{i,j}^a= \left\{\begin{array}{ll}
\bfu_{i,j-a}\,,&j=(a-m) \vee 1,\ldots,(a+m)\wedge p\,,\\
0\,,&\mbox{otherwise\,.}
\end{array}\right.
\end{eqnarray}
The parameter $a$ encodes the location of $\bfu_i$ within $\bfu_i^{a}$. In other words,
\begin{equation}\label{eq:sgdsd}
\bfu_i^a=(0,\ldots,0, \bfu_i',0,\ldots,0)'
\end{equation}
and the location of zeros is determined by $a$. 

\begin{theorem}[Eigenvectors of $\bfP(s_1,s_2)$]\label{thm:eigenvector}
Assume the conditions of Theorem~\ref{thm:mainn} and condition $H_{s_1,s_2}$ with $0\le s_1\le s_2<\infty$. Then for $i\ge 1$,
\begin{equation*}
\ltwonorm{\y_i(s_1,s_2) - \bfu_{b(i)}^{a(i)}(s_1,s_2)} \cip 0\,,\qquad \nto\,.
\end{equation*}
\end{theorem}
For $s_1=s_2=0$, Theorem~\ref{thm:eigenvector} identifies the structure of the eigenvectors of 
the sample covariance matrix $\X\X'$.

The proof of Theorem~\ref{thm:eigenvector}, which heavily relies on Section~\ref{subsec:sketchproofthm:mainn}, is presented in Section~\ref{proofthm:eigenvector}. 

\begin{figure}[h!]
  \centering
\subfigure{
    \includegraphics[trim = 1.5in 3.4in 1.5in 3.6in, clip, scale=0.5]{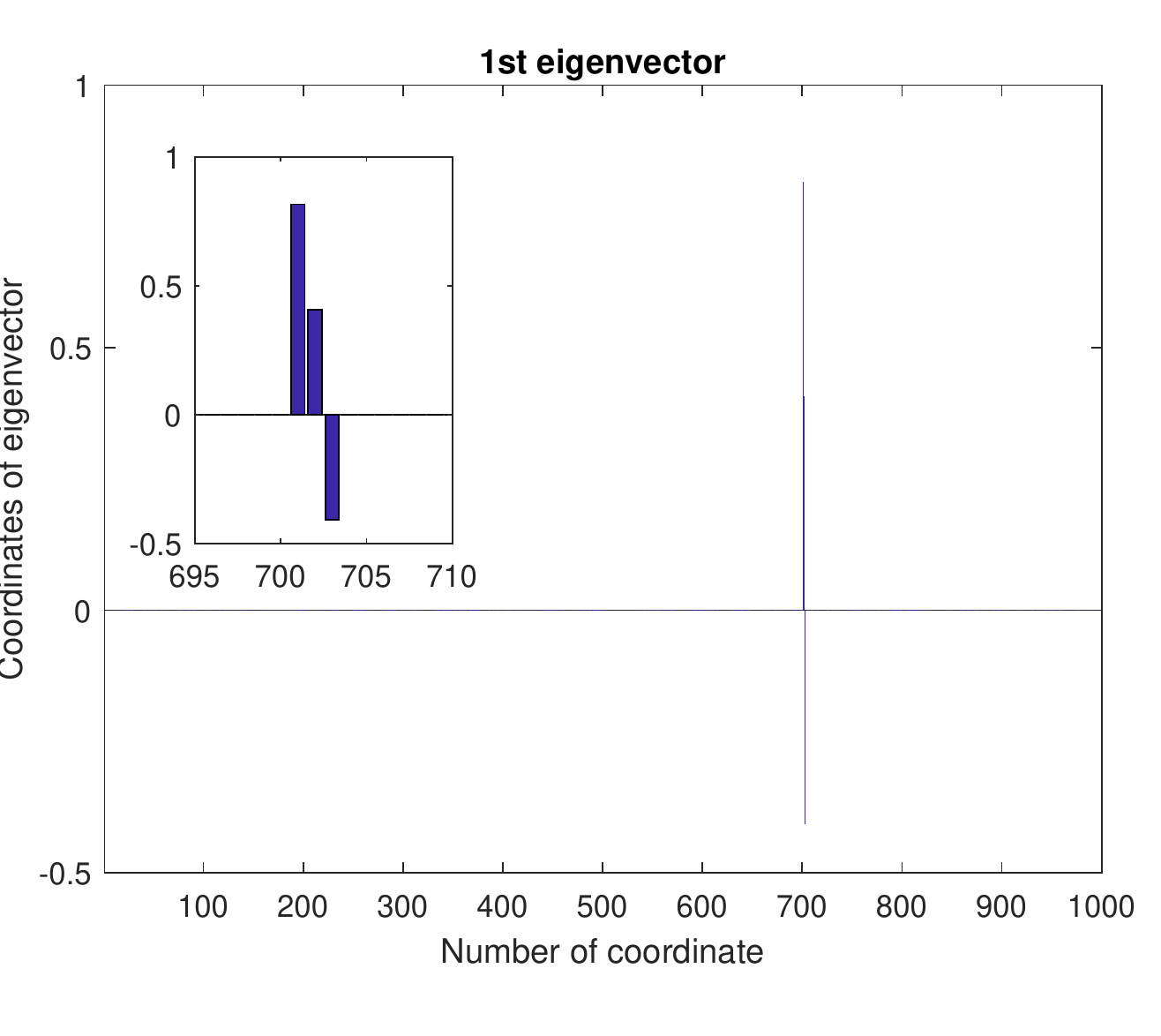}
  }
  \subfigure{
    \includegraphics[trim = 1.5in 3.4in 1.5in 3.6in, clip, scale=0.5]{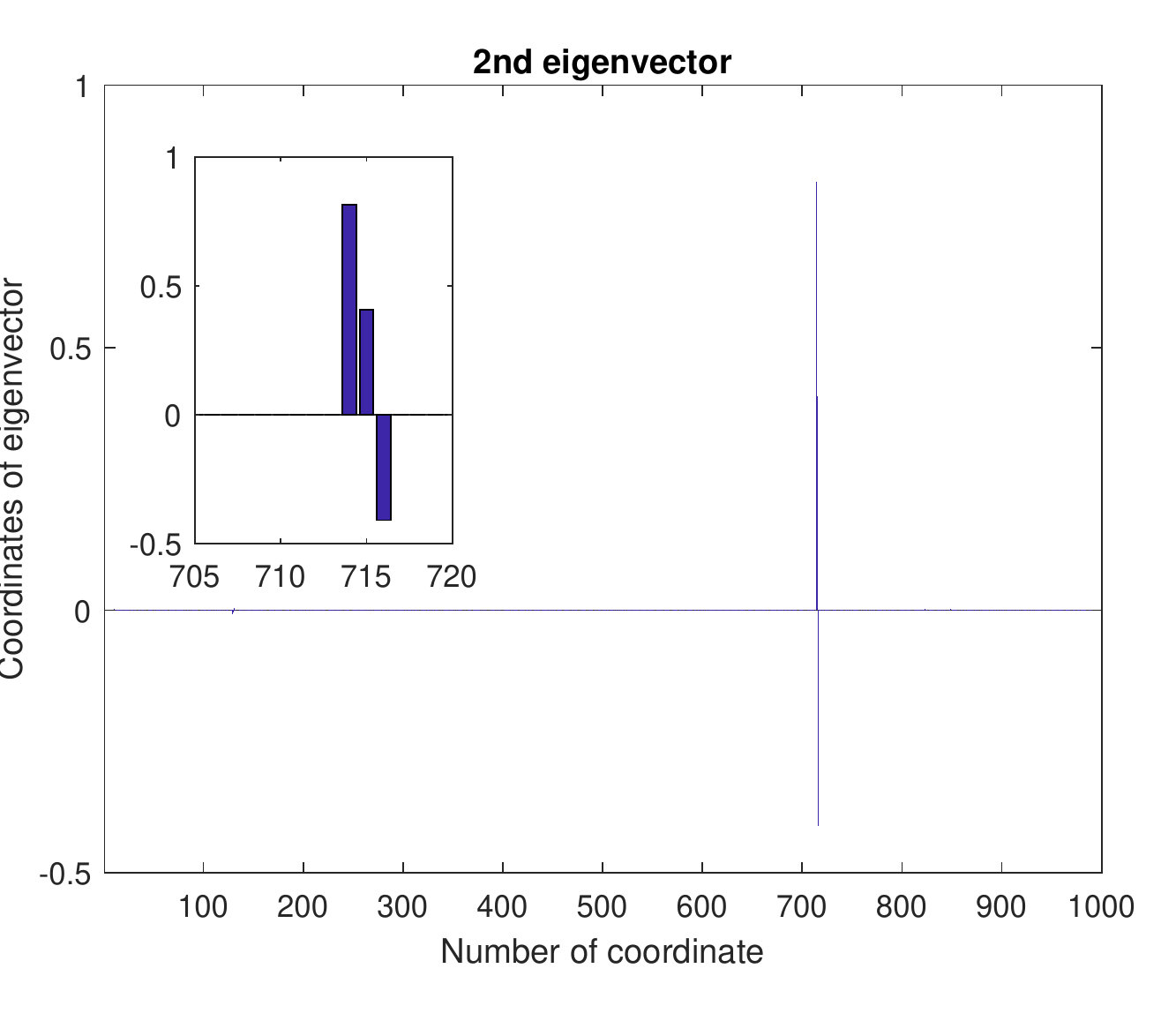}
  }
	\subfigure{
    \includegraphics[trim = 1.5in 3.4in 1.5in 3.6in, clip, scale=0.5]{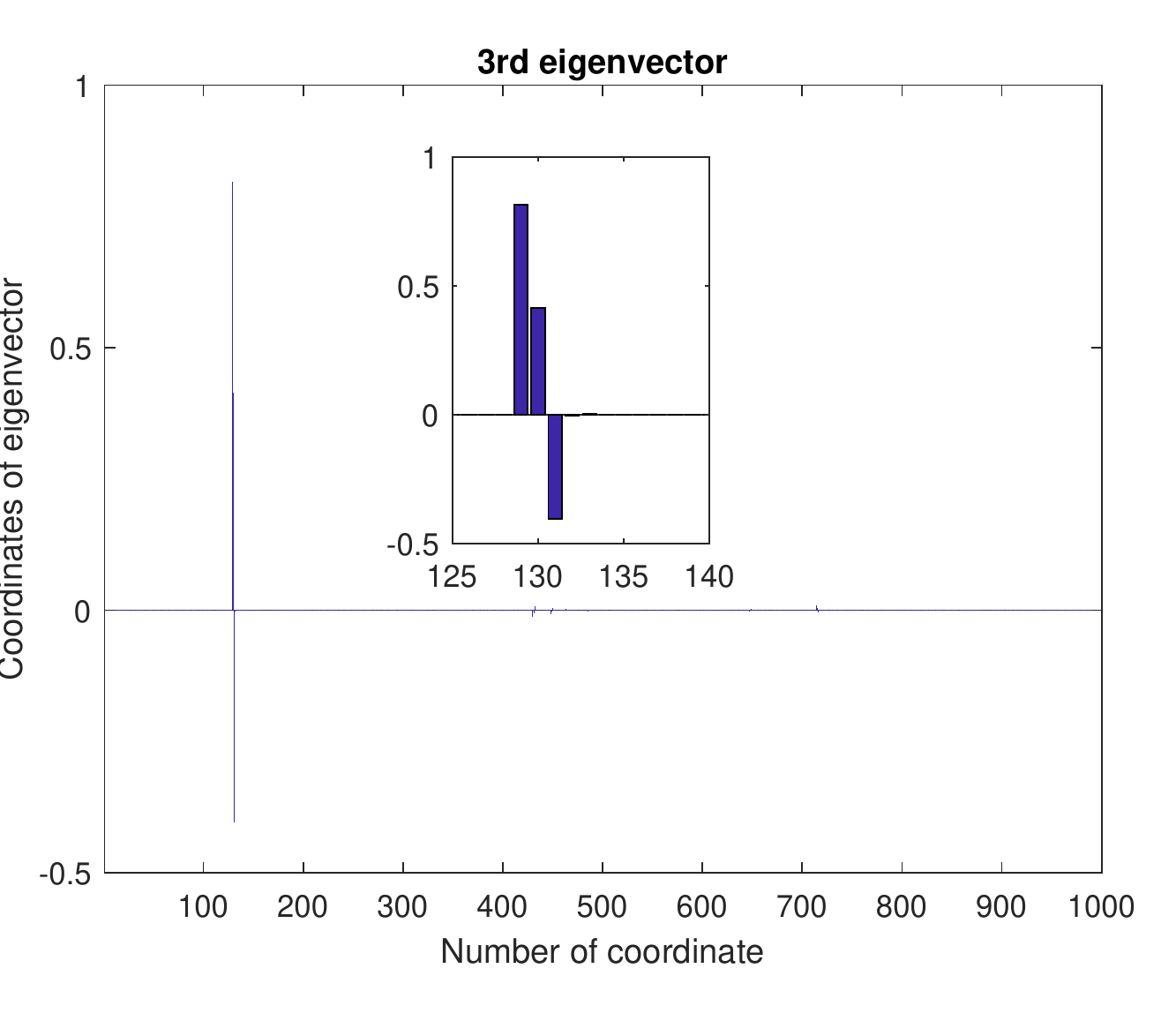}
  }
  \subfigure{
    \includegraphics[trim = 1.5in 3.4in 1.5in 3.6in, clip, scale=0.5]{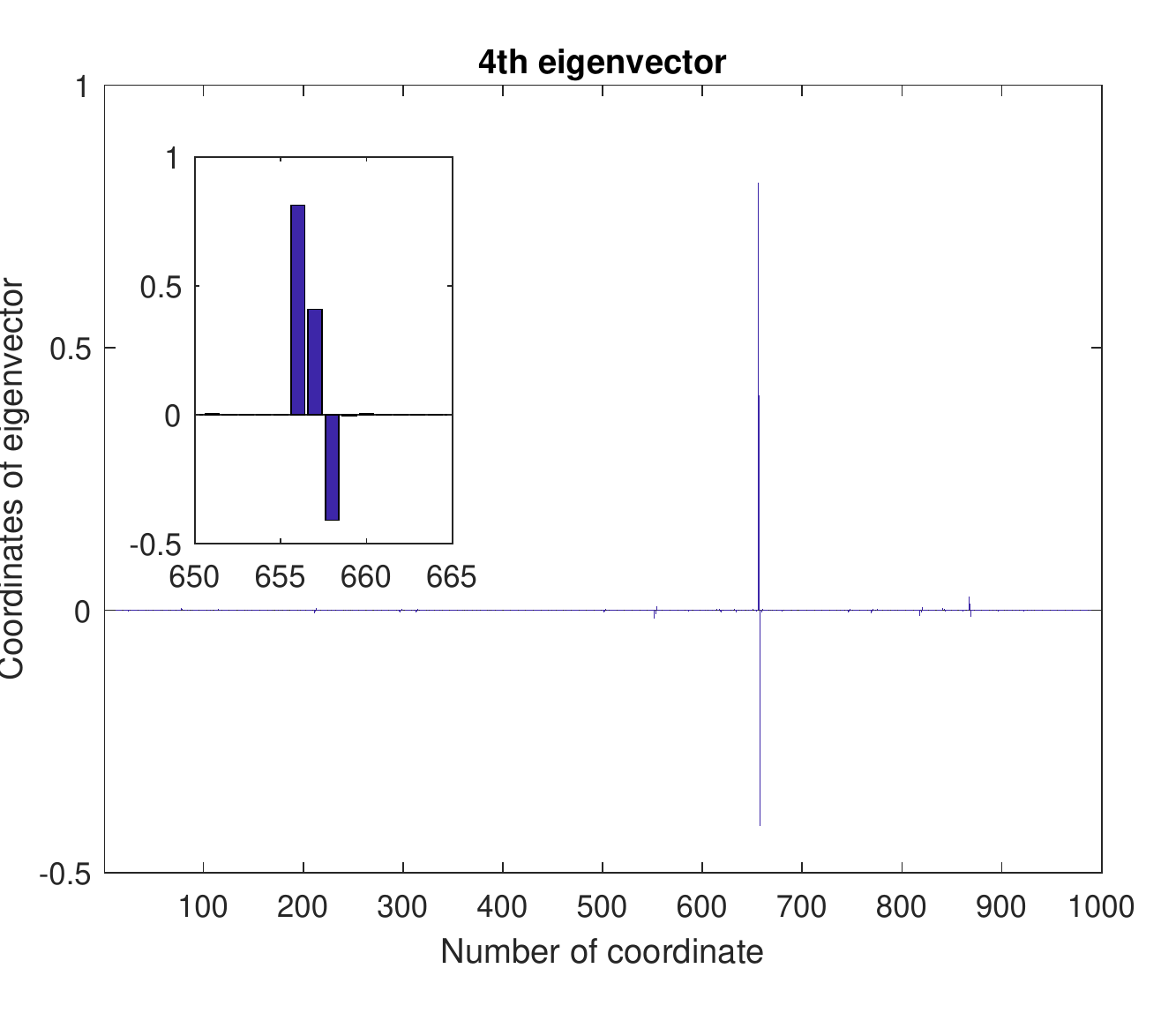}
  }
\caption{Coordinates of the four largest eigenvectors of $\bfP(0,0)$ in the separable case. The eigenvectors predicted by our model are of the form $ 6^{-1/2} \, (2,1,-1)'$. After zooming in, it is easy to see that our approximation is very accurate. The eigenvectors of $\bfP(0,0),\bfP(1,1),\bfP(2,2)$; i.e., for lags 0,1,2; look exactly the same.}
  \label{fig:abc}
\end{figure}

\begin{figure}[h!]
  \centering
\subfigure{
    \includegraphics[trim = 1.5in 3.4in 1.5in 3.6in, clip, scale=0.5]{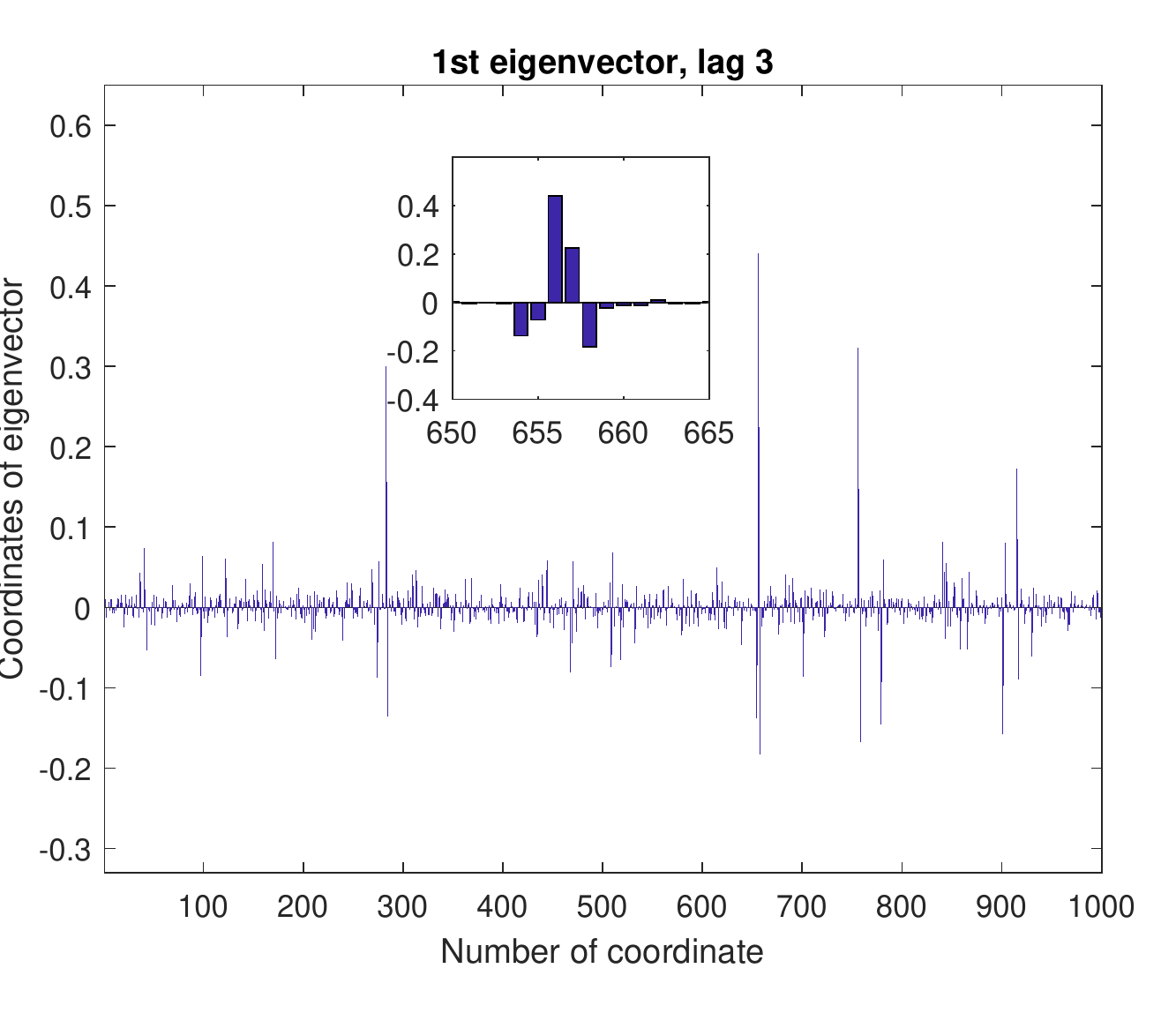}
  }
  \subfigure{
    \includegraphics[trim = 1.5in 3.4in 1.5in 3.6in, clip, scale=0.5]{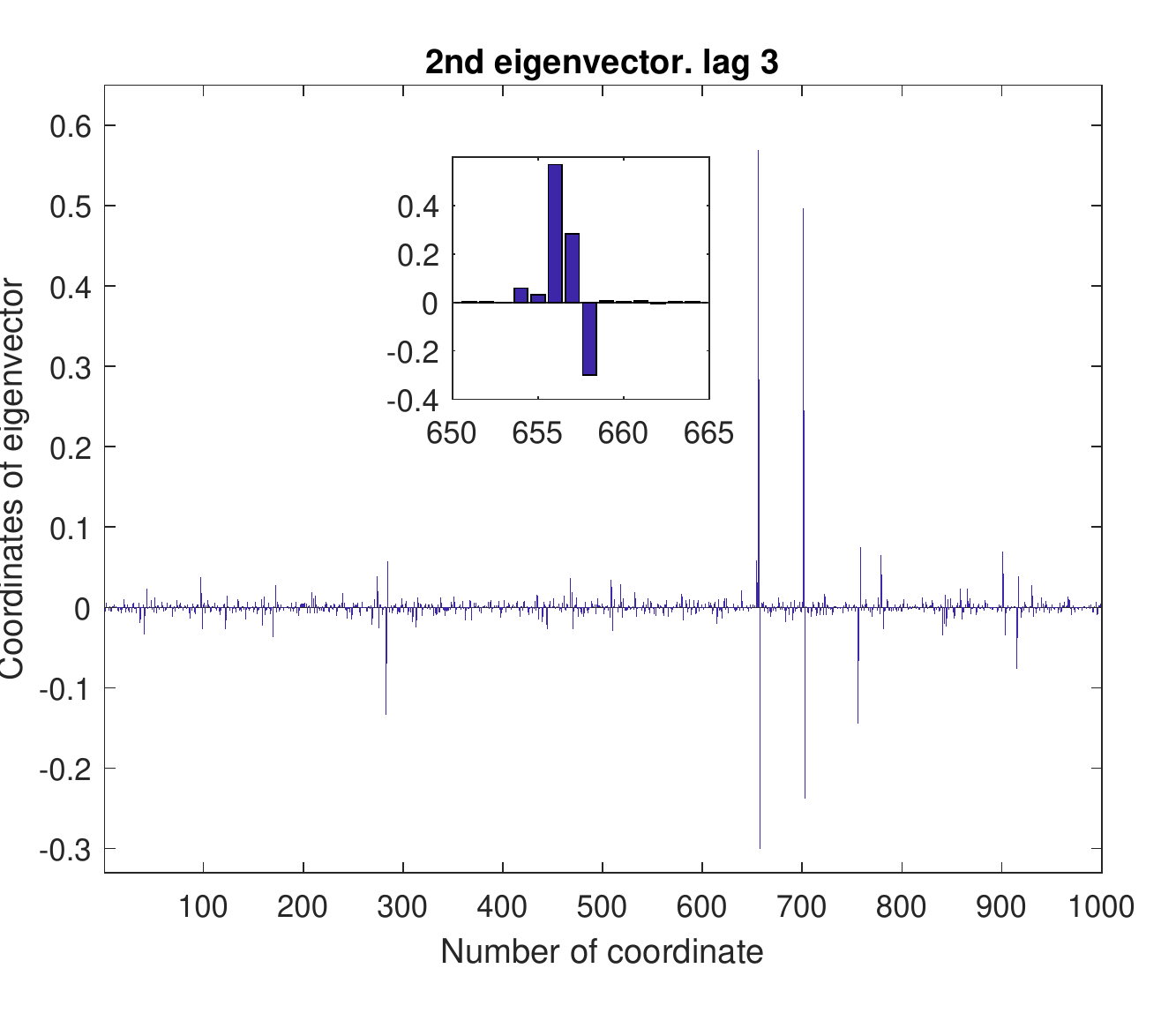}
  }
\caption{Coordinates of the leading eigenvectors of $\bfP(3,3)$ for underlying $\bfK(3,3)$-matrix being zero. The eigenvectors are very different from those of $\bfP(0,0),\bfP(1,1),\bfP(2,2)$; i.e., for lags 0,1,2.}
  \label{fig:abc1}
\end{figure}

\begin{example}\label{ex:eigenvectors}{\em 
We consider the separable case $h_{kl}=d_kc_l$ and re-use the setting of Example~\ref{ex:eigenvalues}. In addition, we assume that $d_k=c_k=0$ for $k>m$. 

Note that it is sufficient to focus on the non-zero elements of $\widehat \bfK$ in \eqref{eq:Khat}. By \eqref{eq:Kex} and symmetry of $\bfD=(d_id_j)$, we obtain $\widehat \bfK(s_1,s_2)= \sum_{s=s_1}^{s_2} \ov c ^2(s) \bfD_m^2$, where $\bfD_m=(d_i d_j)_{0\le i,j\le m}$. 
One easily checks that the $(m+1)$-dimensional unit eigenvector of $\bfD_m$ associated with its only non-zero eigenvalue $\ov d$ is 
\beam\label{eq:u1}
\bfu_1=\ov d^{-1/2} (d_0,\ldots,d_{m})'\,.
\eeam
By assumption $d_0>0$, this vector is oriented in accordance with our convention.
Now we verify that $\bfu_1$ is also the eigenvector associated with the only non-zero eigenvalue $v_1^2(s_1,s_2)$ of $\widehat \bfK(s_1,s_2)$:
\begin{equation*}
\sum_{s=s_1}^{s_2} \ov c ^2(s) \bfD_m^2\, \bfu_1= \sum_{s=s_1}^{s_2} \ov c ^2(s) \ov d \,\bfD_m\, \bfu_1= v_1^2(s_1,s_2)\,\bfu_1\,.
\end{equation*}

By Theorem~\ref{thm:eigenvector}, for fixed $i$, the eigenvector $\y_i(s_1,s_2)$ is approximated by the $p$-dimensional vector that coincides with $\bfu_{1}$ at the $L_i$th to $(L_i+m)$th coordinates 
and has zero entries otherwise, i.e.,
\begin{equation}\label{eq:eigvec}
\bfu_{b(i)}^{a(i)}(s_1,s_2)=\bfu_{1}^{L_i}(s_1,s_2)
=(\underbrace{0,\ldots,0}_{L_i-1}, \bfu_1',0,\ldots,0)'\,.
\end{equation}

We illustrate the approximation of the leading eigenvectors by \eqref{eq:eigvec} in Figure~\ref{fig:abc}. We set $(d_0,d_1,d_2)=(2,1,-1)$ and $c_0=c_1=c_2=1$; all others are $0$. The iid field $(Z_{it})$ is simulated from a $t_{1.5}$ distribution. Then we construct the sample autocovariance matrices $\X_n(s)$, $s\ge 0$, for dimension $p=1000$ and sample size $n=10000$. From \eqref{eq:u1} one has 
\beam\label{eq:u11}
\bfu_1=\dfrac{1}{\sqrt{6}}\, (2,1,-1)'\,.
\eeam
In Figure~\ref{fig:abc}, we plot the 1000 coordinates of the $\ell_2$-normalized eigenvectors associated with the four largest eigenvalues of $\bfP(0,0)$. It is easy to see that these eigenvectors are of the form \eqref{eq:eigvec}. Indeed, their coordinates are zero everywhere except some region which is determined by the location of the 4 largest values in the iid sequence $D_1,\ldots, D_p$, that is $L_1, \ldots,L_4$. The eigenvectors have a spike at $L_i$ of the form $\bfu_1$ which becomes apparent after zooming in around the $L_i$. In this example, the spikes appear at 3 coordinates since $\bfu_1$ is 3-dimensional. By choosing an appropriate sequence $(d_0,\ldots,d_m)$, it is possible to generate arbitrary spikes. 
The eigenvectors of $\bfP(1,1),\bfP(2,2),\bfP(0,2)$; i.e., for lags 1,2 and the sum of lags 0 to 2; look exactly the same as those of $\bfP(0,0)$. This phenomenon is due to independence of the approximate eigenvectors from $(s_1,s_2)$ in \eqref{eq:eigvec}.

For the same data set we show, in Figure~\ref{fig:abc1}, the leading two eigenvectors of $\bfP(3,3)$, the squared sample autocovariance matrix at lag 3. Since $\bfK(3,3)$ is the null-matrix, the assumptions of Theorem~\ref{thm:eigenvector} are violated. We observe that the structure of the plotted eigenvectors is very different from those in Figure~\ref{fig:abc} which correspond to 
non-null $\bfK$-matrices. Recall that the $\bfK$-matrix contains information about the largest entries of the sample autocovariance matrix. More precisely, it describes the entries with the heaviest tail index $\alpha/2$ which dominate the spectral behavior. If $\bfK$ is the null-matrix, it means that all autocovariance entries have the same tail index $\alpha$. Therefore the mass of the eigenvectors is more spread out what we observe in Figure~\ref{fig:abc1}. Moreover, it is interesting to note that, when zooming in around the largest coordinates of the eigenvectors, we see the somewhat familiar pattern of $\bfu_1$, though not as dominant as in Figure~\ref{fig:abc}. 

Next, we present some consequences of identity \eqref{eq:ident} for the sums of eigenvalues and the eigenvalues of sums of matrices. The following table contains the ratios of the largest eigenvalues of $\bfP(s,s)$ and $\bfP(0,0)$ for lags $s=1,\ldots, 5$.  

%\begin{table}
  \begin{center}
    \begin{tabular}{l|c c c c c} % <-- Alignments: 1st column left, 2nd middle and 3rd right, with vertical lines in between
    lag $s$   & 1 & 2 & 3 & 4 & 5\\
      \hline
    $\frac{\la_1(s,s)}{\la_1(0,0)}$ & \, 0.4444 \,& \, 0.1111 \, & \, $3.31 \cdot 10^{-6}$ \, & \, $4.05 \cdot 10^{-6}$ \, & \, $7.43 \cdot 10^{-6}$  \\
    \end{tabular}
  \end{center}
	%\label{table:1}
%\end{table}
The eigenvalue at lag 0 is of highest magnitude which can be explained by the inequality $|\ov c (s)|\le \ov c (0)$. Our approximations for $\la_1(s,s)$, $s\ge 0$ in \eqref{eq:mainn1} are $\gamma_1(s,s)=D_{(1)}^2 \ov c (s)^2 \ov d^2$. In this example we have $\ov c (0)=3, \ov c (1)=2, \ov c (2)=1, \ov c (3)=0$. Hence,
\begin{equation*}
\frac{\la_1(s,s)}{\la_1(0,0)} \approx \frac{\gamma_1(s,s)}{\gamma_1(0,0)} = \frac{\ov c (s)^2}{\ov c (0)^2}\,,
\end{equation*}
which provides a theoretical explanation for the values in the table. From lag 3 onwards, all the entries of the sample autocovariance matrix have lighter tails than some entries of the sample autocovariance matrices with smaller lags.

Finally, we calculate 
\begin{equation*}
\frac{\la_1(0,1)}{\la_1(0,0)}=1.4444 \quad \text{ and } \quad \frac{\la_1(0,2)}{\la_1(0,0)}=1.5555\,.
\end{equation*}
In words, the sum of the largest eigenvalues of $\bfP(0,0)$ and $\bfP(1,1)$ equals the largest eigenvalue of $(\bfP(0,0)+\bfP(1,1))$.
}\end{example}

In Example~\ref{ex:eigenvectors}, all non-null $\bfK(s_1,s_2)$-matrices had rank 1 and the same eigenvector $\bfu_1$ associated with $v_1^2(s_1,s_2)$. If the coefficients $(h_{kl})$ of the linear process \eqref{eq:1}
are such that $\bfK(s_1,s_2)$ has a rank higher than 1, and the eigenvectors depend on the lags $(s_1,s_2)$, one obtains a much richer structure of eigenvectors.

\begin{example}\label{ex:eigenvectors1}{\em 
\begin{figure}[htb!]
  \centering
    \includegraphics[trim = 0.15in 0.9in 0.35in 0.5in, clip, scale=0.75]{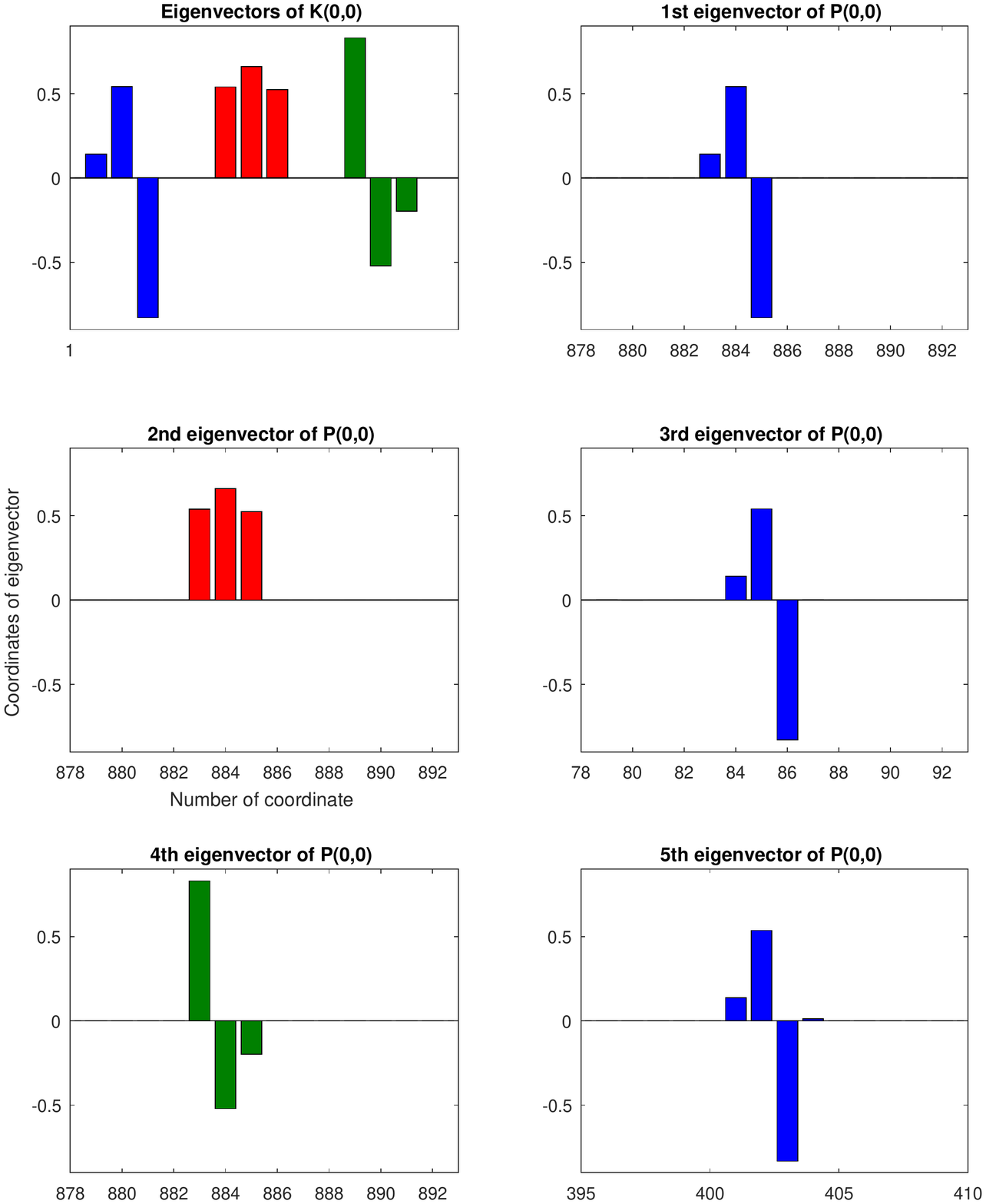}
\caption{Eigenvectors of $\bfP(0,0)$ for a $\bfK$-matrix with rank 3.}
\label{fig:00}
\end{figure}
We study the autocovariance matrices of the linear process \eqref{eq:1}, where $(Z_{it})$ is a field of independent identically $t$-distributed random variables with $1.5$ degrees of freedom.
The coefficients of the linear process are given by
\begin{equation}\label{eq:33}
\begin{pmatrix}
h_{00} & h_{01} & h_{02}\\
h_{10} & h_{11} & h_{12}\\
h_{20} & h_{21} & h_{22}
\end{pmatrix}
=
\begin{pmatrix}
1 & 2 & 0\\
4 & 1 & -1\\
-3 & 0 & 5
\end{pmatrix}\,.
\end{equation}
We work with simulated data $(X_{it})$ for dimension $p=1000$ and sample size $n=10000$. Our goal is to examine the quality of the asymptotic approximation of the $p$-dimensional eigenvectors of $\bfP(0,0)$ and $\bfP(1,1)$ provided by Theorem \ref{thm:eigenvector}.

Since $\bfH$ is zero outside a $3\times 3$ block, the $\bfK$-matrices inherit this property by definition \eqref{eq:Ks1s2} and we will interpret them as $3\times 3$ matrices for simplicity.   

We start with $\bfP(0,0)$. The eigenvalues of the deterministic $\bfK(0,0)$ are
\begin{equation*}
v_1^2(0,0)=2080.1\,,\quad v_2^2(0,0)=89.1 \,,\quad v_3^2(0,0)=3.8\, 
\end{equation*}
with associated (normalized) eigenvectors
\begin{equation*}
\bfu_1(0,0)=
\begin{pmatrix}
0.1412\\
0.5411\\
-0.8290
\end{pmatrix}
\,,\quad
\bfu_2(0,0)=
\begin{pmatrix}
0.5392\\
0.6602\\
0.5228
\end{pmatrix}
\,,\quad
\bfu_3(0,0)=
\begin{pmatrix}
0.8303\\
-0.5208\\
-0.1986
\end{pmatrix}\,,
\end{equation*}
respectively. Recall that all vectors have Euclidean norm 1. The coordinates of $\bfu_1(0,0)$, $\bfu_2(0,0)$, $\bfu_3(0,0)$ are plotted in the top left panel of Figure~\ref{fig:00} in blue, red and green, respectively. By Theorem~\ref{thm:eigenvector} and \eqref{eq:sgdsd}, the eigenvectors of $\bfP(0,0)$ should resemble the appropriately shifted versions of $\bfu_1(0,0), \bfu_2(0,0), \bfu_3(0,0)$. Therefore we compute the eigenvectors of $\bfP(0,0)$ and try to match them with either a blue, red or green pattern from the top left panel. 

The result for the eigenvectors associated with the 5 largest eigenvalues of $\bfP(0,0)$ is presented in Figure~\ref{fig:00}. The top right panel, for instance, shows the coordinates 878 to 893 of the first eigenvector. We zoomed in on the interesting region 878-893 since all other coordinates are very close to zero; compare also with Figure~\ref{fig:abc} where all coordinates and a zoom-in version are plotted. One immediately notices the pattern of $\bfu_1(0,0)$ at location $L_1=883$; see \eqref{eq:help6} for the definition of $L_i$.  The color blue is chosen to emphasize the resemblance of the first eigenvector of $\bfP(0,0)$ to $\bfu_1(0,0)$.

In the second and third rows of Figure~\ref{fig:abc}, we see that the blue, red and green patterns from the top left panel can be easily detected in the eigenvectors of $\bfP(0,0)$.  Since $\bfK(0,0)$ has rank 3, we observe all 3 patterns. 

The zoom-in location is determined by $(L_i)$. In this example, the possible zoom-in locations are $L_1,\ldots, L_5$. Moreover, it is possible to determine the $L_i$'s, which are defined in terms of order statistics of the iid noise, by looking at the eigenvector plots. From Figure~\ref{fig:00} we can deduce the following: the pattern in the first eigenvector is always located at $L_1$ and therefore $L_1=883$. More generally, the largest $L_i$'s can be found by plotting the first, second, third,\ldots eigenvectors of $\bfP(0,0)$. The $k$th appearance of the $\bfu_1(0,0)$ pattern corresponds to $L_k$. An inspection of the second column of Figure~\ref{fig:00} gives $L_1=883, L_2=84$ and $L_3=401$; see also Figure~\ref{fig:11}. 
\par
\begin{figure}[htb!]
  \centering
    \includegraphics[trim = 0.15in 0.9in 0.35in 0.5in, clip, scale=0.75]{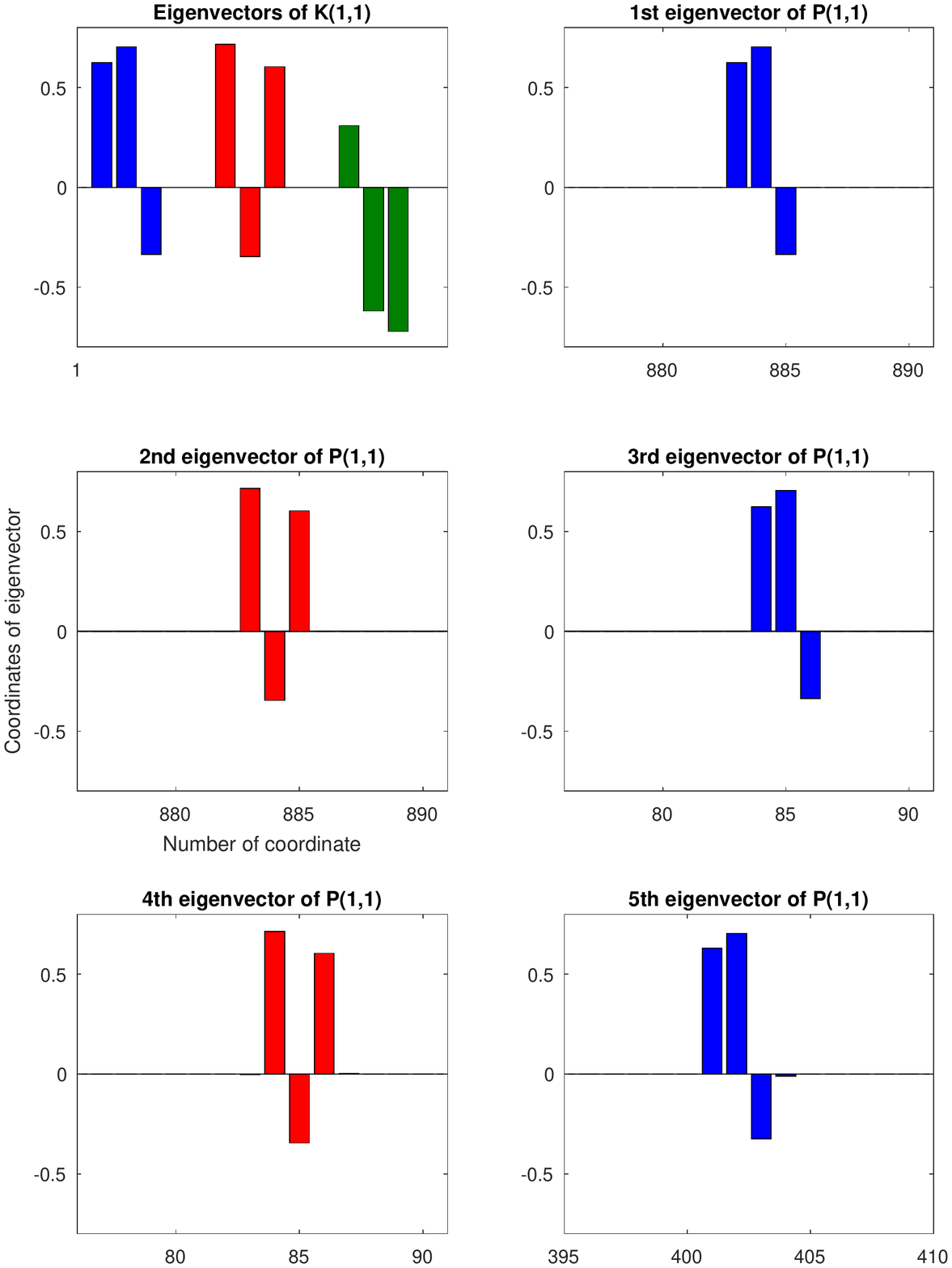}
\caption{Eigenvectors of $\bfP(1,1)$ for a $\bfK$-matrix with rank 2.}
\label{fig:11}
\end{figure}

In Figure~\ref{fig:11} we show the leading eigenvectors of $\bfP(1,1)$. The eigenvalues of the deterministic $\bfK(1,1)$ are $v_1^2(1,1)=181.00, v_2^2(1,1)=66.99$ and $v_3^2(1,1)=0$
with associated (normalized) eigenvectors
\begin{equation*}
\bfu_1(1,1)=
\begin{pmatrix}
0.6242\\
0.7050\\
-0.3368
\end{pmatrix}
\,,\quad
\bfu_2(1,1)=
\begin{pmatrix}
0.7174\\
-0.3465\\
0.6044
\end{pmatrix}
\,,\quad
\bfu_3(1,1)=
\begin{pmatrix}
0.3094\\
-0.6189\\
-0.7220
\end{pmatrix}\,,
\end{equation*}
respectively. The coordinates of $\bfu_1(1,1), \bfu_2(1,1), \bfu_3(1,1)$ are plotted in the top left panel of Figure~\ref{fig:11} in blue, red and green, respectively. 
In contrast to $\bfK(0,0)$, the matrix $\bfK(1,1)$ does not have full rank. As a consequence the green pattern corresponding to $\bfu_3$ does not appear within the eigenvectors of $\bfP(1,1)$ (see Figure~\ref{fig:11}) and we only need to look for the the blue and red patterns. Indeed, in Theorem~\ref{thm:eigenvector} only the eigenvectors $\bfu_i$ associated with positive eigenvalues of $\bfK$ are considered.  

Next, we show the connection between the eigenvector plots and Theorem~\ref{thm:mainn}.
To this end, recall that $\gamma_i(1,1)$ in \eqref{eq:mainn1} is the $i$th largest value in the set $\{ v_j^2(1,1) D_{(\ell)}^2:j,\ell \ge 1\}$. The approximate eigenvectors in Theorem~\ref{thm:eigenvector} are $\bfu_{b(i)(1,1)}^{a(i)(1,1)}(1,1)$, where $a(i)(1,1), b(i)(1,1)$ satisfy the equation $\gamma_{i}(1,1)= D_{a(i)(1,1)}^2\,{v}_{b(i)(1,1)}^2(1,1)$. The $b(i)$'s essentially decide which pattern will be observed in the sample autocovariance eigenvectors. In this example, we have $b(i)(1,1)\in \{1,2\}$. If $b(i)(1,1)=j$, we will see the $\bfu_j(1,1)$ pattern within the coordinates of the $i$th largest eigenvector of $\bfP(1,1)$. Consequently, the $b(i)$'s can be obtained from Figure~\ref{fig:11}. We immediately find
\begin{equation*}
b(1)(1,1)=1, \quad b(2)(1,1)=2, \quad
b(3)(1,1)=1, \quad
b(4)(1,1)=2, \quad
b(5)(1,1)=1\,.
\end{equation*}
Similarly, from Figure~\ref{fig:00} we have
\begin{equation*}
b(1)(0,0)=1, \quad b(2)(0,0)=2, \quad
b(3)(0,0)=1, \quad
b(4)(0,0)=3, \quad
b(5)(0,0)=1\,.
\end{equation*}
}\end{example}
To summarize, the eigenvector plots contain a lot of information about the model. 
The location of the spikes provides insight into the structure of the iid noise, while the spikes themselves can be viewed as functions of the $(h_{kl})$. More precisely, the vectors $\bfu_i(s_1,s_2)$ are functions of the coefficients $(h_{kl})$. If the number of non-zero coefficients is small, the $\bfu_i(s_1,s_2)$ can be estimated from the eigenvectors of $\bfP(s_1,s_2)$. Doing so for various pairs $(s_1,s_2)$, it is possible to invert the  functional relation between $\bfu_i(s_1,s_2)$ and the coefficients. Hence, one can estimate the coefficients $(h_{kl})$ of the linear process.

While Example~\ref{ex:eigenvectors1} treated the case of spikes of length 3, one easily obtains spikes of length $m$ by replacing \eqref{eq:33} by a matrix with $m$ rows. The computational effort would stay the same provided $m$ is small relative to the dimension $p$. This setting is quite natural for factor models.

%--------------------------------------------------------------
\subsection{Point process \con }\label{sec:pp}
Theorem~\ref{thm:mainn}  and arguments similar to the proofs in \cite{davis:heiny:mikosch:xie:2016}
enable one to derive the weak \con\ of
the point processes of the normalized eigenvalues of $\bfP(s_1,s_2)$:
\begin{equation}\label{eq:ppdef}
N_n^{(s_1,s_2)}=\sum_{i=1}^p \delta_{a_{np}^{-4}\lambda_{i}(s_1,s_2)}\,,
\end{equation}
where $\delta_x$ denotes the Dirac measure at $x$.
 
\begin{theorem}\label{thm:pp}
Assume the conditions of Theorem~\ref{thm:mainn}.
Then $(N_n^{(s_1,s_2)})$ converge weakly in the space of point measures
with state space $(0,\infty)$ equipped with the vague topology:
\begin{equation}\label{eq:pp}
N_n^{(s_1,s_2)} \cid  N^{(s_1,s_2)}= \sum_{i=1}^\infty
\sum_{j=1}^{\infty} \delta_{\Gamma_i^{-4/\alpha} v_j^2(s_1,s_2)}\,, \qquad \nto.
\end{equation}
Here
\beao
\Gamma_i=E_1+\cdots + E_i\,,\qquad i\ge 1\,,
\eeao
and $(E_i)$ is an iid standard exponential \seq .
\end{theorem}
For the proof of Theorem~\ref{thm:pp} one can follow the lines of the proof of 
Theorem 3.4 in \cite{davis:heiny:mikosch:xie:2016}; we omit further details.
\begin{remark}\rm 
Since \eqref{eq:pp} is based on a continuous mapping
argument involving the same points $\Gamma_i^{-4/\alpha}$ for arbitrary $(s_1,s_2)$, \eqref{eq:pp}
immediately extends to processes whose multivariate points 
consist of $a_{np}^{-4}\la_i(s_1,s_2)$
for distinct choices of $(s_1,s_2)$.
\end{remark} 

\par
The limiting point process in \eqref{eq:pp} yields a plethora of ancillary results. For example, one can easily derive the limiting distribution of $a_{np}^{-4} \lambda_{k}(s_1,s_2)$ for fixed $k\ge 1$: 
\begin{equation*}
\begin{split}
\lim_{\nto}\P(a_{np}^{-4} \lambda_{k}(s_1,s_2)\le x)&= \lim_{\nto}\P(N_n^{(s_1,s_2)}(x,\infty)<k)
=  \P(N^{(s_1,s_2)}(x,\infty)<k)\,, \quad  x>0\,.
\end{split}
\end{equation*}
In particular,the normalized  largest eigenvalue has a re-scaled $\Phi_{\alpha/4}$-Fr\'echet
\ds :
\begin{equation*}
\begin{split}
\lim_{\nto}\P(a_{np}^{-4} \lambda_1(s_1,s_2)\le x)&=
\P(\Gamma_1 ^{-4/\alpha} v_1^2(s_1,s_2)\le x)=  \exp(-x^{-\alpha/4} v_1^{\alpha/2}(s_1,s_2)), \quad  x>0\,.
\end{split}
\end{equation*}
\par
In the paper \cite{lam:yao:2012}
the following estimator was considered in the context of a factor model for $s_1=1$, and fixed values 
$s_2,k\ge 1$:
\beam\label{eq:lamyaoest}
\underset{1\le i\le k}{\mathrm{arg\,min}}\,  \dfrac{\la_{i+1}(s_1,s_2)}{\la_{i}(s_1,s_2)}\,.
\eeam
Writing $q_i$ for the $i$th largest value in the set $\{ v_j^2(s_1,s_2) \Gamma_{\ell}^{-4/\alpha}:j,\ell \ge 1\}$, we have the joint \con\ of the ratios
\beam\label{eq:lamyaoestlim}
\Big(\dfrac{\la_2(s_1,s_2)}{\la_1(s_1,s_2)},\ldots,\dfrac{\la_{k+1}(s_1,s_2)}{\la_k(s_1,s_2)}\Big)
\std \Big( \dfrac{q_2}{q_1},\ldots,
  \dfrac{q_{k+1}}{q_k}   \Big)\,.
\eeam
Hence the limit \ds\ for the estimator \eqref{eq:lamyaoest} can be achieved by
simulation.
\par
In particular, if $r(s_1,s_2)=1$,
 another immediate consequence of \eqref{eq:pp} is
\beao
\frac{ \big(\la_{1}(s_1,s_2),\ldots,\la_{k+1}(s_1,s_2)\big)}{a_{np}^{4} \, v_1^{2} (s_1,s_2)}\std 
\big(\Gamma_1^{-4/\alpha},\ldots,\Gamma_{k+1}^{-4/\alpha}\big)\,,
\eeao
so that \eqref{eq:lamyaoestlim} reads as 
\beao
\Big(\dfrac{\la_2(s_1,s_2)}{\la_1(s_1,s_2)},\ldots,\dfrac{\la_{k+1}(s_1,s_2)}{\la_k(s_1,s_2)}\Big)
\std \Big( \Big(\dfrac{\Gamma_1}{\Gamma_2}\Big)^{4/\alpha},\ldots,
  \Big(\dfrac{\Gamma_{k}}{\Gamma_{k+1}}\Big)^{4/\alpha}   \Big)\,.
\eeao
Hence the limit \ds\ would not depend on $(s_1,s_2)$ in this case. Notice that
the ratios $\Gamma_i/\Gamma_{i+1}$ 
are independent Beta$(i,1)$-distributed for $i=1,\ldots,k$. Therefore
\beao
-\Big(\log \dfrac{\la_2(s_1,s_2)}{\la_1(s_1,s_2)}, 2\,\log \dfrac{\la_3(s_1,s_2)}{\la_2(s_1,s_2)},\ldots,k\,\log\dfrac{\la_{k+1}(s_1,s_2)}{\la_k(s_1,s_2)}\Big)
\std \dfrac 4 \alpha \,(E_1,\ldots,E_k)\,.
\eeao
%and the limit of the log-spacings {\red 
%\beao
%&&\big(\log \la_{1}(s_1,s_2)-\log \la_{2}(s_1,s_2),\ldots,\log \la_{k+1}(s_1,s_2)-\log \la_{k}(s_1,s_2)\big)\\&\std&
%-\dfrac{4}{\alpha}\,\big(\log(\Gamma_1/\Gamma_2),\ldots,\log (\Gamma_{k}/\Gamma_{k+1})\big)\,.
%\eeao
%We also observe that the limit \ds\ of the estimator \eqref{eq:lamyaoest} does not depend on $s$ in this case.
\par
A continuous mapping argument similar to \cite{resnick:2007}, Theorem~7.1,
yields for the trace
\beao
a_{np}^{-4}\,{\rm tr}\big( \bfP(s_1,s_2)\big)=a_{np}^{-4}\sum_{i=1}^p\la_i(s_1,s_2)\std 
\sum_{j=1}^\infty v_j^{2} (s_1,s_2) \sum_{i=1}^\infty \Gamma_i^{-4/\alpha}\,.
\eeao
Since $\alpha\in (0,4)$, the right-hand series converges a.s. and represents a positive  
$\alpha/4$-stable \rv ; see  \cite{samorodnitsky:taqqu:1994} for more information on series representations 
of stable \rv s. We also have the joint \con\ 
\beao
a_{np}^{-4} \Big(\la_{1}(s_1,s_2)\,,{\rm tr}\big( \bfP(s_1,s_2)\big)\Big)\std 
\Big(v_1^2(s_1,s_2)\Gamma_1^{-4/\alpha}, \sum_{j=1}^\infty v_j^{2} (s_1,s_2) \sum_{i=1}^\infty \Gamma_i^{-4/\alpha}\Big)\,.
\eeao
Therefore we have  self-normalized \con\ of the largest eigenvalue  $\la_{1}(s_1,s_2)$
\beao
\dfrac{\la_{1}(s_1,s_2)}{{\rm tr}\big( \bfP(s_1,s_2)\big)} \std
\dfrac{v_1^2(s_1,s_2)}{\sum_{j=1}^\infty v_j^{2} (s_1,s_2)}\dfrac{\Gamma_1^{-4/\alpha} }{ \sum_{i=1}^\infty \Gamma_i^{-4/\alpha} }\,.
\eeao
The limiting variable is the scaled  quotient of a $\Phi_{\alpha/4}$-Fr\'echet \rv\ and a positive 
$\alpha/4$-stable \rv .

%---------------------------------------------------------------------------
\subsection{Singular values of the symmetrization}\label{subsec:34}
For $s \ge 1$, the sample autocovariance matrix $\bfC(s)$ may have complex eigenvalues. An alternative way of creating real eigenvalues is by applying {\em symmetrization}. Therefore we 
study the matrix 
\beam\label{eq:sss}
\bfA_n(s_1,s_2)=\sum_{s=s_1}^{s_2} \mbox{$\frac 12$}\, \big(\bfC_n(s) +\bfC_n(s)'\big)\,
\eeam
and its singular values 
\begin{equation*}
\sigma_{1}(s_1,s_2)\ge \cdots \ge \sigma_{p}(s_1,s_2)\ge 0\,.
\end{equation*}
Our focus is on singular values because the eigenvalues can be negative. This corresponds to an ordering of eigenvalues with respect to their absolute values.

The role of the matrix $\bfK(s_1,s_2)$ in \eqref{eq:Ks1s2} will now be played by
\begin{equation}\label{eq:Ks1s2t}
\wt\bfK(s_1,s_2)= \sum_{s=s_1}^{s_2} \mbox{$\frac 12$}\, \big(\bfM(s) +\bfM(s)'\big)
\end{equation}
with ordered singular values $\wt v_1(s_1,s_2) \ge \wt v_2(s_1,s_2) \ge\cdots $. Again we assume that $\wt\bfK(s_1,s_2)$ is not the null-matrix. Then we have the following analog of 
Theorem~\ref{thm:mainn}.

\begin{theorem}\label{thm:mainan}
Assume the conditions of Theorem~\ref{thm:mainn}.
Then we have for $0\le s_1\le s_2<\infty$,
\begin{equation}\label{eq:mainsym}
a_{np}^{-4} \max_{i=1,\ldots,p} |\sigma_{i}(s_1,s_2)-\wt\gamma_{i}(s_1,s_2)| \cip 0, \quad \nto\,,
\end{equation}
where $(\wt\gamma_{i}(s_1,s_2))$ are the ordered values of the set
$\big\{D_{(i)} \wt v_j(s_1,s_2)\,, i=1,\ldots,p\,;j=1,2,\ldots\big\}$.

Moreover, if $\alpha<2(1+\beta)$, then
\begin{equation*}
a_{np}^{-4} \max_{i=1,\ldots,p} |\sigma_{i}(s_1,s_2)-\wt\delta_{i}(s_1,s_2)| \cip 0, \quad \nto\,,
\end{equation*}
where $(\wt\delta_{i}(s_1,s_2))$ are the ordered values of the set
$\big\{Z_{(i),np}^2 \wt v_j(s_1,s_2)\,, i=1,\ldots,p\,;j=1,2,\ldots\big\}$.
\end{theorem}
An inspection of the proof of Theorem~\ref{thm:mainn} shows that all its parts can be modified  
when $\bfP_n(s_1,s_2)$ is replaced by $\bfA_n(s_1,s_2)$; therefore we omit a proof.
Theorem~\ref{thm:eigenvector} also remains valid for the eigenvectors $\y_i(s_1,s_2)$ of $\bfA_n(s_1,s_2)$ if we let the $\bfu_i(s_1,s_2)$ denote the eigenvectors of $\wt\bfK(s_1,s_2)$.

As an analog of Theorem \ref{thm:pp} we get
\begin{equation*}
\sum_{i=1}^p \delta_{a_{np}^{-2}\sigma_{i}(s_1,s_2)}
 \cid  \sum_{i=1}^\infty
\sum_{j=1}^{\infty} \delta_{\Gamma_i^{-2/\alpha} \wt v_j(s_1,s_2)}\,, \qquad \nto.
\end{equation*}

\begin{example}\label{ex:symmetrization}{\em 
Again, we consider the separable case and use the notation and assumptions from Examples \ref{ex:eigenvalues} and \ref{ex:eigenvectors}.

By symmetry of $\bfD$, we find
\begin{equation*}
\wt\bfK(s_1,s_2)= \sum_{s=s_1}^{s_2} \mbox{$\frac 12$}\, \big(\bfM(s) +\bfM(s)'\big) 
= \sum_{s=s_1}^{s_2} \ov c (s) \bfD\,.
\end{equation*}
The approximating values in \eqref{eq:mainsym} are therefore 
$\wt \gamma_i(s_1,s_2)=D_{(i)} | \sum_{s=s_1}^{s_2} \ov c (s)| \, \, \ov d$, $1\le i\le p$.
In contrast to the equality \eqref{eq:ident}, we only obtain the inequality 
\begin{equation}\label{eq:ident2}
\wt \gamma_i(s_1,s_2) \le \sum_{s=s_1}^{s_2} \wt \gamma_i(s,s)\,
\end{equation}
for the symmetrized autocovariances.
This is due to the fact that $\ov c (s)$ can be positive or negative. Different signs lead to a cancellation effect which reduces the magnitude of the singular values.

Since $|\ov c (s)|\le \ov c (0)$ for $s>1$ we also observe that the approximating quantities $\wt \gamma_i(s,s)$
are always dominated by  $\wt \gamma_i(0,0)$.
Using that $(\ov c (s))$ is a non-negative definite \fct , we conclude that
$\sum_{s=0}^{s_0}\ov c (s)\ge 0$, hinting at the fact that lag $0$ is of central importance.

Finally, if the $\bfc$-\seq\ has non-negative components, then there is equality in \eqref{eq:ident2} which implies $\wt \sigma_i(s_1,s_2)\approx \sum_{s=s_1}^{s_2} \wt \sigma(s,s)$ for large $n$. It is apparent in Figure~\ref{fig:LamYao2}(b)  that this approximation does not necessarily hold for real-life return data.

An approximation to the eigenvector of $\bfA_n(s_1,s_2)$ associated with the $i$th largest absolute eigenvalue is given by \eqref{eq:eigvec}.

As regards point processes, we have 
\beao
\sum_{i=1}^p \delta_{a_{np}^{-2} \,\wt v_1^{-1}(s_1,s_2) \, \sigma_i(s_1,s_2)} \std \sum_{i=1}^\infty \delta_{\Gamma_i^{-2/\alpha}}\,,\qquad \nto\,.
\eeao
In other words, the limit is a Poisson point process on $(0,\infty)$ with mean \ms\
$\mu(x,\infty)= x^{-\alpha/2}$, $x>0$.

\begin{figure}[htb!]
  \centering
  \subfigure[] {
    \includegraphics[scale=0.4]{eigen_sum.pdf}
    \label{fig:LamYao:a}
  }
  \subfigure[] {
    \includegraphics[scale=0.4]{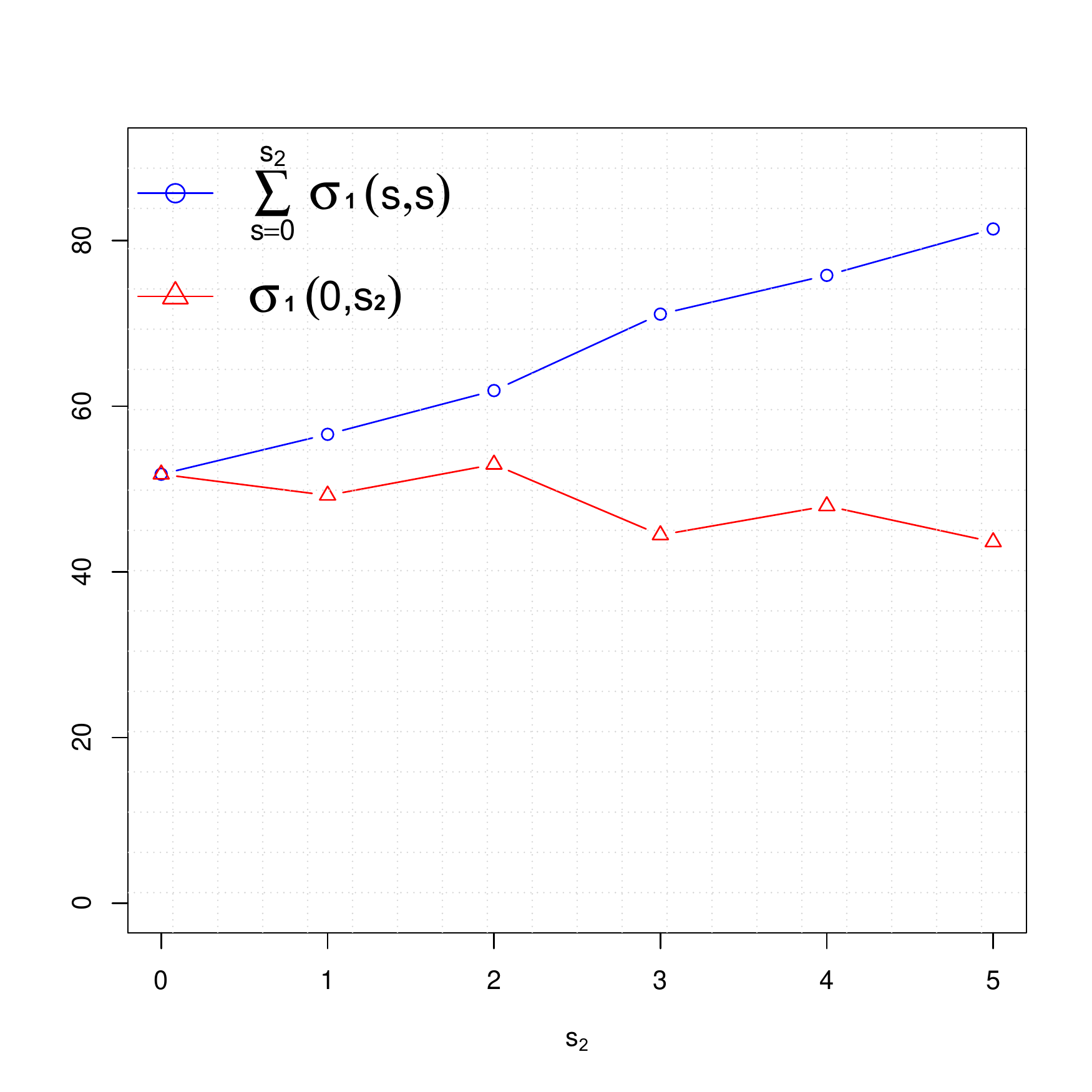}
    \label{fig:LamYao:b}
  }
  \caption{(a): A comparison of the sums of the largest eigenvalues  $\lambda_1(s,s)$ of the squared sample autocovariance matrices and the largest eigenvalue $\lambda_1(0,s_2)$ of the sum of these matrices for $s_2=0,\ldots,5$.
The underlying data $\bfX$ consists of $p=478$  log-return series composing
the S\&P 500 index estimated from $n=1345$ daily observations from 01/04/2010 to 02/28/2015.
(b): We compare the sum of the largest largest singular values of the symmetrized sample autocovariance matrices and the largest singular value of their sum. The matrices in the latter sum might not be positive definite which is a possible explanation for the observed cancellation effect.}
  \label{fig:LamYao2}
\end{figure}
}\end{example}

%-------------------------------------------
\subsection{Limiting spectral distribution}\label{sec:limit}

So far we studied the behaviors of the largest eigenvalues of sample autocovariance matrices. In Section~\ref{sec:pp}, we employed point process techniques to describe their joint convergence. We saw that they are separated from each other, which in turn enabled us to characterize the associated eigenvectors in Section~\ref{subsec:32}. In contrast, the bulk (or non-extreme) eigenvalues are usually not separated. The bulk is often studied via the so--called empirical spectral distribution  
which is defined for a $p\times p$ matrix $\bfA$ with real eigenvalues $\lambda_1(\bfA),\ldots,\lambda_p(\bfA)$  by 
\begin{equation*}
F_{\bfA}(x)= \frac{1}{p}\; \sum_{i=1}^p \1_{\{ \lambda_i(\bfA)\le x \}}, \qquad x\in  \R\,.
\end{equation*}
The empirical spectral distribution is uniquely characterized by its Stieltjes transform
\begin{equation*}
s_{\bfA}(z)= \int_{\R} \frac{1}{x-z} \dint F_{\bfA}(x) \,, \quad z\in \C^+\,,
\end{equation*}
where $\C^+$ denotes the complex numbers with positive imaginary part; see for instance \cite{yao:zheng:bai:2015}.

In this subsection, we describe the limit of the empirical spectral distributions of the sample covariance matrices when $p$ and $n$ grow proportionally. An application of Corollary 7 in \cite{banna:merlevede:peligrad:2015} yields the following result.
\begin{proposition}
Consider the linear process \eqref{eq:1} and assume that
\begin{itemize}
\item
$p/n \to \gamma \in (0,\infty)$,
\item
$\E[Z^2]=1$ and $\E[Z]=0$,
\item
the summability condition $\sum_{k,l \in \Z}  |h_{kl}| <\infty$ holds.
\end{itemize}
Then the empirical spectral distributions $F_{n^{-1} \X\X'}$ converge, with probability 1, to a nonrandom distribution function $F$ whose Stieltjes transform $s$ satisfies 
\begin{equation}\label{eq:sz}
s(z)=\int_0^1 h(x,z) \dint x\,, \quad z\in \C^+\,,
\end{equation}
where $h(x,z)$ is a solution to the equation 
\begin{equation*}
h(x,z)=\Big( -z +\int_0^1 \frac{f(x,t)}{1+\gamma \int_0^1 f(u,t) h(u,z) \dint u} \dint t  \Big)^{-1}
\end{equation*}
with
\begin{equation*}
f(x,y)=\sum_{k,l\in \Z} \gamma_{kl} \e^{-2\pi \mathbf{i} (kx+ly)}\quad \text{ and } \quad \gamma_{kl}= \sum_{u,v\in \Z} h_{uv} h_{u-k,v-l}\,.
\end{equation*}
\end{proposition}

The distribution function $F$ can be obtained numerically from \eqref{eq:sz}; we refer to \cite{dobriban:2015} for details.
In the iid case, i.e. $X_{it}=Z_{it}$, \eqref{eq:sz} reads as $s(z)=\frac{1+\gamma s(z)}{-z -z \gamma z s(z)+1}$
with solution  
\beam\label{eq:stieltjestransform}
s_{F_{\gamma}}(z)
= \frac{1-\gamma -z +\sqrt{(1+\gamma-z)^2-4\gamma}}{2 \gamma z} \,, \quad z\in \C^+\,.
%+ \frac{1-\gamma}{\gamma z} \1_{\{\gamma>1\}}\,.
\eeam
This is the Stieltjes transform of the famous Mar\v cenko--Pastur law $F_\gamma$. If $\gamma \in (0,1]$,  $F_\gamma$  has density,
\begin{eqnarray}\label{eq:MPch1}
f_\gamma(x) =
\left\{\begin{array}{ll}
\frac{1}{2\pi x\gamma} \sqrt{(b-x)(x-a)} \,, & \mbox{if } a\le x \le b, \\
 0 \,, & \mbox{otherwise,}
\end{array}\right.
\end{eqnarray}\noindent
where $a=(1-\sqrt{\gamma})^2$ and $b=(1+\sqrt{\gamma})^2$. If $\gamma>1$, the \MP law has an additional point mass $1-1/\gamma$ at $0$.

\begin{remark}{\em
There is a minor typo in Corollary 7 of \cite{banna:merlevede:peligrad:2015}. In the definition of  $\mathbb{B}_N$, $N^{-1}$ needs to be replaced by $p^{-1}$; compare with equation (11) of the same paper.
}\end{remark}

%-----------------------------------------------------------------------------
\section{Proofs}\label{sec:proof}\setcounter{equation}{0}

%--------------------------------------------------
\subsection{Sketch of the proof of Theorem~\ref{thm:mainn} in the case of a finite filter}\label{subsec:sketchproofthm:mainn}

We consider a finite filter $(h_{kl})$ in the sense that for some $m\in \N$ we have $h_{kl}=0$ if $|k|\vee |l|>m$. This implies that $\bfM(s)_{ij}=0$ for $|i|\vee |j|>m$ and all $s\ge 0$. 
In words, $\bfM(s)$ is zero outside of a block of size $(2m+1)\times (2m+1)$. Now we embed this block into $p\times p$ matrices $\bfM_{a}(s) =(\bfM_{a}(s)_{ij})_{i,j=1,\ldots,p}$, $a\in\Z$, which we define by  

\begin{eqnarray}\label{eq:matrixm}
{\bfM}_{a}(s)_{ij}= \left\{\begin{array}{ll}
{\bfM}(s)_{i-a,j-a}\,,&i,j=(a-m) \vee 1,\ldots,(a+m)\wedge p\,,\\
0\,,&\mbox{otherwise.}
\end{array}\right.
\end{eqnarray}

%In words, ${\bfM}_{a}(s)$ has zero entries but the $(2m+1)\times (2m+1)$ block with lower left corner $(a-m,a-m)$ which  coincides with ${\bfM}(s)$. 
For $m+1\le a\le p-m$, ${\bfM}(s)$ 
has the following block-diagonal form
\begin{equation*}
\bfM(s) =
\begin{pmatrix}
\zero & & \\
& \bfM_a(s) &\\
& & \zero
\end{pmatrix}\,.
\end{equation*}  
%Therefore, $\bfM(s)$ and  $\bfM_a(s)$ have the same non-zero singular values.

Recall that $\zero$ denotes a quadratic matrix consisting of zeros. 

 By Theorem~\ref{prop:main} below, we have for any integer sequence $k=k_p\to \infty$ such that $k_p^2=o(p)$,
\begin{equation}\label{eq:appfinite}
a_{np}^{-2} \Big\|  \bfC_n (s)- \sum_{i=1}^k D_{(i)} \bfM_{L_i}(s) \Big\|_2 \cip 0\,, \quad \nto \,.
\end{equation}
For such a $k$ it holds $\P(A_n)\to 1$, where  
\begin{equation}\label{eq:setA}
A_n=\{|L_i-L_j|>2m+1\,, i,j=1,\ldots,k\,, i\ne j\} \cap \{ 1\le L_i-m \le L_i+m\le p \,, i=1,\ldots,k \}\,.
\end{equation}
On the set $A_n$, we have
\begin{equation}\label{eq:nonoverlap}
\bfM_{L_i}(s) \bfM_{L_j}(s)'= \zero \,, \quad i,j=1,\ldots,k\,, i\ne j\,.
\end{equation}
A combination of \eqref{eq:appfinite} and \eqref{eq:nonoverlap} shows 
\begin{equation}\label{eq:appfinite1}
a_{np}^{-4} \Big\| \bfP(s_1,s_2)- \sum_{i=1}^k D_{(i)}^2 \sum_{s=s_1}^{s_2} \bfM_{L_i}(s) \bfM_{L_i}(s)' \Big\|_2 \cip 0\,, \quad \nto \,.
\end{equation}
Define $\bfK_a(s_1,s_2)$ analogously to \eqref{eq:matrixm}. Then we get the identity 
\begin{equation*}
\bfK_{L_i}(s_1,s_2) = \sum_{s=s_1}^{s_2} \bfM_{L_i}(s) \bfM_{L_i}(s)'\,, \quad i=1,\ldots,k\,.
\end{equation*}
The eigenvalues of the block-diagonal matrix $\sum_{i=1}^k D_{(i)}^2 \bfK_{L_i}(s_1,s_2)$, which approximates $\bfP(s_1,s_2)$, are the $(\gamma_{i}(s_1,s_2))$; consult the proof of Theorem~\ref{thm:eigenvector} for more insight. Finally, an application of Weyl's theorem \cite{bhatia:1997} on eigenvalue perturbations finishes the proof of \eqref{eq:mainn1}. A detailed proof can be found in Section~\ref{subsec:proofthm:mainn}.

%--------------------------------------------------
\subsection{Proof of Theorem~\ref{thm:eigenvector}}\label{proofthm:eigenvector}

In this proof we suppress the dependence of most quantities on $(s_1,s_2)$ in the notation. Let $k=k_p\to\infty$ be an integer sequence such that $k_p^2=o(p)$ as $\nto$. We have seen in \eqref{eq:appfinite1} that 
$\sum_{i=1}^k D_{(i)}^2 \bfK_{L_i}$ approximates $\bfP$ in spectral norm. Our proof consists of two steps:
\begin{itemize}
\item[(i)] Show that the eigenvectors of $\sum_{i=1}^k D_{(i)}^2 \bfK_{L_i}$ associated with its largest eigenvalues are given by $\bfu_{b(i)}^{a(i)},i\ge 1$.
\item[(ii)] Bound the difference between the eigenvectors of $\bfP$ and those of $\sum_{i=1}^k D_{(i)}^2 \bfK_{L_i}$.
\end{itemize}

For our considerations it is sufficient to work on the set $A_n$ defined in \eqref{eq:setA}. On $A_n$, we have the following block-diagonal structure of $\sum_{i=1}^k D_{(i)}^2 \bfK_{L_i}=\sum_{i=1}^k D_{L_i}^2 \bfK_{L_i}$:
\begin{equation}\label{eq:form}
\bfU_n:=\sum_{i=1}^k D_{L_i}^2 \bfK_{L_i}=
\begin{pmatrix}
\zero &&&&&\\
& D_{\pi_1}^2 \widehat \bfK &&&&\\
&& \zero &&&\\
&&& \ddots &&\\
&&&& D_{\pi_k}^2 \widehat \bfK&\\
&&&&& \zero
\end{pmatrix}\,,
\end{equation} 
where $\pi_1,\ldots,\pi_k$ is a certain permutation of $L_1,\ldots,L_k$; see \eqref{eq:help6}. 
From \eqref{eq:form} one deduces that the  eigenvectors associated with the positive eigenvalues of $\bfU_n$ are given by the ``appropriately shifted'' eigenvectors $\bfu_i$ of $\widehat \bfK$ and must be of the form \eqref{eq:sgdsd}.
The positive eigenvalues of $\bfU_n$ are $D_{\pi_i}^2 \, v_j^2$ with associated eigenvectors 
$\bfu_j^{\pi_i} \,, i\le k, j\le r$.  

On the set $A_n$, we have for fixed $1 \le j \le k$, noting that $a(j)=L_i$ for some $i$,
\begin{equation*}
\begin{split}
\Big( \sum_{i=1}^k D_{L_i}^2 \bfK_{L_i} \Big) \bfu_{b(j)}^{a(j)}&=\sum_{i=1, L_i\neq a(j)}^k D_{L_i}^2 \bfK_{L_i} \bfu_{b(j)}^{a(j)} + D^2_{a(j)} \bfK_{a(j)}\bfu_{b(j)}^{a(j)}\\
&= D^2_{a(j)} {v}^2_{b(j)} \bfu_{b(j)}^{a(j)} =  \gamma_j  \bfu_{b(j)}^{a(j)}.
\end{split}
\end{equation*}
Therefore, the eigenvector of $\bfU_n$ associated with its $j$th largest eigenvalue $\gamma_j$ is $\bfu_{b(j)}^{a(j)}$. This  finishes the proof of step~(i).
\par

Next, we turn to step (ii). By definition of the spectral norm as a supremum over the unit sphere and \eqref{eq:appfinite1}, we have for $\vep>0$, 
\begin{equation*}
\begin{split}
\lim_{\nto} &\P\Big(a_{np}^{-4} \max_{i=1,\ldots,k}\max_{j=1,\ldots,r} \ltwonorm{\bfP \, \bfu_j^{L_i}-D^2_{L_i} v_j^2 \, \bfu_j^{L_i}} > \vep \Big)\\ 
&\le \lim_{\nto} \P\Big( a_{np}^{-4}\Big\| \bfP-\sum_{i=1}^k
 D^2_{L_i} \bfK_{L_i}\Big\|_2 > \vep \Big) = 0\,.
\end{split}
\end{equation*}
This shows that for $j\ge 1$ fixed,
\begin{equation}\label{eq:mhjf}
\vep^{(n)}:=a_{np}^{-2} \ltwonorm{\bfP \, \bfu_{b(j)}^{a(j)}-\gamma_j \, \bfu_{b(j)}^{a(j)}} \cip 0\,.
\end{equation}

Before we can apply Proposition~\ref{prop:perturbation} we need to show that, 
with probability converging to $1$, there are no other eigenvalues in a suitably small interval around $\lambda_j$.
By assumption $H_{s_1,s_2}$, 
any two non-zero eigenvalues of $\bfK$ are distinct. Hence, recalling that $r$ is the rank of $\bfK$,
\begin{equation}\label{eq:helpnow}
v_j^2>v_{j+1}^2\,, \quad j\le r.
\end{equation}
Let $\xi >1$. We define the set
\begin{equation*}
\Omega_n = \Omega_n(j,\xi)= \{a_{np}^{-4} |\la_j-\la_i|>\xi\, \vep^{(n)}\, :\, i\neq j =1,\ldots,p  \}\,.
\end{equation*} 
Using \eqref{eq:mhjf} and \eqref{eq:helpnow} combined with Theorem~\ref{thm:mainn}, we obtain
\begin{equation*}
\begin{split}
\lim_{\nto} \P\big(\Omega_n^c) &= \lim_{\nto} \P( a_{np}^{-4} \min\{ \la_{j-1}-\la_{j}, \la_j-\la_{j+1}\} \le \xi\, \vep^{(n)}\big) =0\,.
\end{split}
\end{equation*}
%The separation of the eigenvalues $(v_{j}^2)$ implies the separation of the eigenvalues $(\wt \lambda_{\wt L_i})$.

By Proposition~\ref{prop:perturbation}, the unit eigenvector $\y_j$ associated with $\la_j$ and the projection $\Proj_{\bfu_{b(j)}^{a(j)}}(\y_{j})$
of  the vector $\y_{j}$ onto the linear space generated by $\bfu_{b(j)}^{a(j)}$
satisfy for fixed $\delta>0$:
\begin{equation*}
\begin{split}
\limsup_{\nto} \P\big(\ltwonorm{\y_j-\Proj_{\bfu_{b(j)}^{a(j)}}(\y_{j})}>\delta\big) &
\le \limsup_{\nto} \P(\{\ltwonorm{\y_j-\Proj_{\bfu_{b(j)}^{a(j)}}(\y_{j})}>\delta\} \cap \Omega_n \cap A_n)\\
&\quad + \limsup_{\nto}\P((\Omega_n \cap A_n)^c)\\
&\le \limsup_{\nto} \P(\{2\vep^{(n)}/(\xi\,\vep^{(n)}-\vep^{(n)})>\delta\} \cap \Omega_n\cap A_n)\\
&\le \limsup_{\nto} \P(\{2/(\xi-1)>\delta\})=  \1_{\{ 2/(\xi-1)>\delta\}}.
\end{split}
\end{equation*}
The right-hand side is zero for sufficiently large $\xi$. Since both $\y_j$ and $\bfu_{b(j),a(j)}$ are unit vectors 
and $\ltwonorm{\bfP_{\bfu_{b(j)}^{a(j)}}(\y_{j})}\le 1$, this means that
$\ltwonorm{\y_j - \bfu_{b(j)}^{a(j)}} \cip 0\,.$ The proof is complete.

\subsection{Preliminaries for the proof of Theorem~\ref{thm:mainn}}\label{sec:proof1}

To handle the case of an infinite filter $(h_{kl})$ we introduce a truncation of the matrix $\bfM(s)$. For $s,m\ge 0$, define the $(2m+1)\times (2m+1)$ matrix 
\begin{equation*}
\bfM^{(m)}(s)= \Big( \sum_{l=0}^m h_{il} h_{j,l+s} \Big)_{i,j=-m,\ldots,m}\,
\end{equation*}
with  rank $r_m(s)$  and ordered singular values
\beam\label{eq:v11}
v_1^{(m)}(s)\ge\cdots \ge v_{2m+1}^{(m)}(s)\,.
\eeam
\begin{remark}\rm 
In the case of a finite filter, the equality
\begin{equation*}
\bfM(s) =
\begin{pmatrix}
\zero & & \\
& \bfM^{(m)}(s) &\\
& & \zero
\end{pmatrix}\,
\end{equation*}  
holds for a sufficiently large $m$. Also note that then $v_j^{(m)}(s)=v_j(s),j\le r(s)=r_m(s)$ with $v_j(s)$ defined below \eqref{eq:v1}. Therefore our notation in \eqref{eq:v11} is consistent with \eqref{eq:v1}. 
\end{remark}

In analogy to \eqref{eq:matrixm}, we embed the small matrix $\bfM^{(m)}(s)$ into $p\times p$ matrices 
$\bfM_{a}^{(m)}(s) =(\bfM^{(m)}_{a}(s)_{ij})_{i,j=1,\ldots,p}$, $a\in\Z$, which we define by
\begin{eqnarray}\label{eq:matrixmtr}
{\bfM}_{a}^{(m)}(s)_{ij}= \left\{\begin{array}{ll}
{\bfM}^{(m)}(s)_{i-a,j-a}\,,&i,j=(a-m) \vee 1,\ldots,(a+m)\wedge p\,,\\
0\,,&\mbox{otherwise.}
\end{array}\right.
\end{eqnarray}
We note that, for $m+1\le a\le p-m$, $\bfM_a^{(m)}(s)$ has rank $r_m(s)$ and the same non-zero singular values as $ \bfM^{(m)}(s)$.
\par
The following result is key to the proof of Theorem~\ref{thm:mainn}. 
\begin{theorem}\label{prop:main}
Consider the linear process \eqref{eq:1} under the assumptions of Theorem~\ref{thm:mainn}.
Then the following statement holds for $s\ge 0$ and any integer sequence $k=k_p\to \infty$ such that $k_p^2=o(p)$:
\begin{equation}\label{eq:appr1}
\lim_{m\to \infty} \limsup_{\nto} \P \Big( a_{np}^{-2} \Big\|  \bfC_n(s)- \sum_{i=1}^k D_{(i)} \bfM_{L_i}^{(m)}(s) \Big\|_2 >\vep  \Big)=0\,, \quad \vep >0\,.
\end{equation}
\end{theorem}
We provide the proof of Theorem~\ref{prop:main} in Section~\ref{subsec:Thm6.2} after the proof of Theorem~\ref{thm:mainn}.
\begin{remark}\label{rem:beta>1}{\em
It is possible % -- and in principle not difficult -- 
 to present most of our results on eigenvalues also for $\beta>1$. The main idea is to use 
the fact that the non-zero eigenvalues of the matrices $\X(0)\X(s)'$ and $\X(s)'\X(0)$ are the same. 
However, the statement of the theorems alone would require a significant amount of additional notation
and therefore we restrict ourselves to $\beta\in [0,1]$.
 As an illustration we formulate an analog of Theorem~\ref{prop:main} where we now assume that $\beta>1$. 
Then if $\alpha< 2(1+\beta^{-1})$, we have
\begin{equation}\label{eq:appr1a}
\lim_{m\to \infty} \limsup_{\nto} \P \Big( a_{np}^{-2} \Big\| \X(s)' \X(0)- \sum_{i=1}^k D^{\downarrow}_{(i)} \bfM_{L^{\downarrow}_i}^{\downarrow,(m)}(s) \Big\|_2 >\vep  \Big)=0\,, \quad \vep >0\,,
\end{equation}
where $\bfM(s)^{\downarrow}=\bfH(s)' \bfH(0)$ and 
\begin{equation*}
D^{\downarrow}_{(1)}=D^{\downarrow}_{L^{\downarrow}_1} \ge \cdots \ge D^{\downarrow}_{(n)}=D^{\downarrow}_{L^{\downarrow}_n}
\end{equation*}
are the order statistics of the column-sums
\begin{equation*}
D_t^\downarrow=D_t^{(n),\downarrow}= \sum_{i=1}^p Z_{it}^2\,,\qquad t=1,\ldots,n;\; \quad p=1,2,\ldots\,
\end{equation*}
\par
%{\red We also mention that all results for sample autocovariance matrices 
%can be extended to the 
%case of \regvary\ noise $(Z_{it})$ when $\alpha\in (4,8)$. 
%However, in this case all quantities $D_i$ and the approximated eigen- and sing%ular values need to be centered. This is analogous to sample covariance matrice%s 
%for $\alpha\in (2,4)$; we refer to
%\cite{davis:heiny:mikosch:xie:2016,davis:mikosch:pfaffel:2016,heiny:mikosch:201%7:iid} to get an impression of the techniques to be used in this case. We omit further details.}}
%We restrict ourselves to the case $\beta \in [0,1]$.
}
\end{remark}

%--------------------------------------------------
\subsection{Proof of Theorem~\ref{thm:mainn}}\label{subsec:proofthm:mainn}

Let $0\le s_1\le s_2<\infty$ and $s\ge 0$. First, we derive an approximation of $\bfP(s_1,s_2)$.
Let $k=k_p\to \infty$ be an integer sequence such that $k_p^2=o(p)$.
On $A_n$ defined in \eqref{eq:setA}, the matrix 
$\sum_{i=1}^k D_{(i)} \bfM_{L_i}^{(m)}(s)$ is block diagonal and therefore
\begin{equation*}
\begin{split}
\Big(\sum_{i=1}^k D_{(i)} \bfM_{L_i}^{(m)}(s)\Big)\Big(\sum_{i=1}^k D_{(i)} \bfM_{L_i}^{(m)}(s)\Big)'
&= \sum_{i=1}^k D_{(i)}^2  \bfM_{L_i}^{(m)}(s) \bfM_{L_i}^{(m)}(s)'
= \sum_{i=1}^k D_{(i)}^2  \bfK_{L_i}^{(m)}(s,s) 
\end{split}
\end{equation*}	
remains block diagonal; see also \eqref{eq:nonoverlap}. 
Here
\begin{equation*}
\bfK_{L_i}^{(m)}(s_1,s_2) = \sum_{s=s_1}^{s_2} \bfM_{L_i}^{(m)}(s) \bfM_{L_i}^{(m)}(s)'\,, \quad i=1,\ldots,k\,.
\end{equation*}  
We observe that for any fixed $x>0$, with $c(m)= \|\bfM_{L_1}^{(m)}(s)\|_2$,
\beao 
\lefteqn{\P\Big(a_{np}^{-2} \Big\| \sum_{i=1}^k D_{(i)} \bfM_{L_i}^{(m)}(s)  \Big\|_2
>x\Big)}\\
&\le &\P\Big(A_n\cap\Big\{a_{np}^{-2} \Big\| \sum_{i=1}^k D_{(i)} \bfM_{L_i}^{(m)}(s)  \Big\|_2
>x\Big\}\Big)+ \P(A_n^c)\\
&=&\P\Big(A_n\cap\Big\{a_{np}^{-2} D_{(1)} c(m) >x\Big\}\Big)+o(1)\\
&\le &\P(a_{np}^{-2} D_{(1)} c(m)  >x)+o(1)\,.
\eeao
We know from \cite{heiny:mikosch:2017:iid} that 
\beam\label{eq:frechet}
a_{np}^{-2}\,D_{(1)} \std \Gamma_1^{-\alpha/2}\,,
\eeam
where the right-hand variable is Fr\'echet $\Phi_{\alpha/2}$-distributed. 
Moreover, by \eqref{eq:2a},
we have $\limsup_{m\to\infty}c(m)<\infty$.
Fix any $\delta\in (0,1)$. Then we can find a constant $x_0$ \st\
\beam\label{eq:tighta}
\lim_{m\to\infty} \limsup_{\nto} \P\Big(a_{np}^{-2} \Big\| \sum_{i=1}^k D_{(i)} \bfM_{L_i}^{(m)}(s)  \Big\|_2
>x_0\Big)<\delta\,.
\eeam
In view of  Theorem~\ref{prop:main} we can choose $x_0$ \st\
\beam\label{eq:tightb}
\limsup_{\nto}\P(a_{np}^{-2} \twonorm{\bfC_n(s)} >x_0)<\delta\,.
\eeam
Now applications of Theorem~\ref{prop:main},
the triangle inequality and the tightness relations \eqref{eq:tighta}
and \eqref{eq:tightb} yield
%\begin{equation*}
%a_{np}^{-2} \twonorm{\bfC(s)} \quad \text{ and } 
%\quad a_{np}^{-2} \Big\| \sum_{i=1}^k D_{(i)} \bfM_{L_i}^{(m)}(s)  \Big\|_2
%\end{equation*}
%converge in distribution to the same Fr\'{e}chet-type random variable with parameter $\alpha/2$ yields the following approximation of $\bfP(s,s)$:

\begin{equation*}
\lim_{m\to \infty} \limsup_{\nto} \P \Big( a_{np}^{-4} \Big\| \bfP(s,s)- \sum_{i=1}^k D_{(i)}^2 \bfK_{L_i}^{(m)}(s,s) \Big\|_2 >\vep  \Big)=0\,, \quad \vep >0\,.
\end{equation*}
Summation over $s$ gives
\begin{equation}\label{eq:appr123}
\lim_{m\to \infty} \limsup_{\nto} \P \Big( a_{np}^{-4} \Big\| \bfP(s_1,s_2)- \sum_{i=1}^k D_{(i)}^2 \bfK_{L_i}^{(m)}(s_1,s_2) \Big\|_2 >\vep  \Big)=0\,, \quad \vep >0\,,
\end{equation}
which is the analog of \eqref{eq:appfinite1} in the case of an infinite filter.
\par

On $A_n$, the eigenvalues of $\sum_{i=1}^k D_{(i)} \bfK_{L_i}^{(m)}(s_1,s_2)$ are the $p$ largest values in the set
\begin{equation}\label{eq:drghrf}
\{D_{L_i}^2 (v_j^{(m)}(s_1,s_2))^2=D_{(i)}^2 (v_j^{(m)}(s_1,s_2))^2:  i=1,\ldots,k, j=1,\ldots,2m+1\}\cup\{0\}\,,
\end{equation}
where $(v_1^{(m)}(s_1,s_2))^2\ge \cdots \ge (v_{2m+1}^{(m)}(s_1,s_2))^2$ are the largest eigenvalues of $\bfK_{L_1}^{(m)}(s_1,s_2)$.

Because of $a_{np}^{-2} \Dr_{(k)}\cip 0$  (see \cite{heiny:mikosch:2017:iid}) we can write $i=1, \ldots, p$ in \eqref{eq:drghrf}.
The corresponding largest $p$ ordered values of them are denoted by
$\gamma_{1}^{(m)}(s_1,s_2)\ge \cdots\ge \gamma_{p}^{(m)}(s_1,s_2)$.
Combining \eqref{eq:appr123} with Weyl's eigenvalue perturbation inequality (see Bhatia \cite{bhatia:1997}) and recalling that $\P(A_n^c)\to 0$ as $\nto$, we have
\begin{equation}
\lim_{m\to\infty}\limsup_{\nto}\P\Big(a_{np}^{-4}
\max_{i\le p}\Big|
\lambda_{i}(s_1,s_2)-\gamma_{i}^{(m)}(s_1,s_2)\Big|>\vep\Big)=0\,,\quad \vep>0\,.
\end{equation}
Finally, we observe that
\begin{equation*}
a_{np}^{-4}\max_{i\le p} \big|\gamma_{i}^{(m)}(s_1,s_2)-\gamma_{i}(s_1,s_2) \big|\le
a_{np}^{-4}\max_{i\le p} D_i^2\; \max_{i\le p}\big|(v_i^{(m)}(s_1,s_2))^2-v_i^2 (s_1,s_2)\big|=o_\P(1) \,,
\end{equation*}
since  \eqref{eq:frechet} holds
%$a_{np}^{-2}\max_{i\le p} D_i \cid \Gamma_1^{-2/\alpha}$
and $(v_i^{(m)}(s_1,s_2))^2 \to v_i^2 (s_1,s_2)$ uniformly in $i$ because both sequences are monotone. This proves \eqref{eq:mainn1}.

The additional step in the case $\alpha < 2(1+\beta)$
 is to replace $D_{(i)}$ by $Z_{(i),np}^2$. We see that
\begin{equation*}
\max_{j\in \N} v_j(s_1,s_2) \, \max_{i\le p} a_{np}^{-2}\big|D_{(i)}-Z_{(i),np}^2 \big|= c a_{np}^{-2}\big|D_{(i)}-Z_{(i),np}^2 \big| =o_\P(1),
\end{equation*}
as shown in \cite{heiny:mikosch:2017:iid}. This proves \eqref{eq:mainn2} and finishes the proof of Theorem~\ref{thm:mainn}.

%--------------------------------------------------------------
\subsection{Proof of Theorem~\ref{prop:main}}\label{subsec:Thm6.2}

For the ease of presentation we assume $h_{kl}=0$ if $\min(k,l)<0$; the extension to general two-sided filters is straightforward. Note that then the entries of the matrix $\bfM^{(m)}(s)$ are zero outside a block of size $(m+1) \times (m+1)$. 

We use the notation 
\begin{eqnarray}\label{eq:Ztilde}
\Zt_{it} =
\left\{\begin{array}{ll}
Z^2_{it} \,, & \mbox{if } \alpha<2(1+\beta), \\
Z^2_{it} - \E[Z^2]  \,, & \mbox{if } \alpha>2(1+\beta),
\end{array}\right.\qquad i,t\in \Z\,.
\end{eqnarray}\noindent
\subsection*{Reduction of $\bfC_n(s)$ to the matrix of the sums of squares}\label{sec:proofs}
The following matrix contains all squared elements of $\bfC_n(s)$ for $s\ge 0$:
%\beao
%\bfQr_n(s)_{ij}=\sum_{k=0}^\infty
%\sum_{l=0}^\infty h_{kl}h_{k+j-i,l+s}\sum_{t=1}^n Z_{i-k,t-l}^2 \,.
%\eeao
\begin{equation*}
\bfQr_n(s)_{ij}=\sum_{l=0}^\infty
\sum_{k=0}^\infty h_{k,l}h_{k+j-i,l+s}\sum_{t=1}^n \Zt_{i-k,t-l}\,.
\end{equation*}
Our goal is to show that the \asy\ properties of the singular values of $\bfC_n(s)$ are
determined by $\bfQr_n(s)$. 
\begin{proposition}\label{prop:withoutQ}
Assume the conditions of Theorem~\ref{thm:mainn} and $s\ge 0$. Then  we have
\beao
a_{np}^{-2} \twonorm{\bfC_n(s) - \bfQr_n(s)} \cip 0\,,\qquad\nto\,.
\eeao
\end{proposition}
Proposition \ref{prop:withoutQ} has the interpretation that the squared $Z$'s with the heaviest tails dominate the spectral behavior of $\bfC_n(s)$. For a more detailed explanation involving large deviation theory we refer to the comments below Proposition 2.2 in \cite{davis:heiny:mikosch:xie:2016}.
\begin{proof}%[Proof of Proposition~\ref{prop:withoutQ}]
First, we observe that
\begin{equation*}
\begin{split}
\bfC_n(s) - \bfQr_n(s) &= \X(0)\X(s)'-\Big(\sum_{l=0}^\infty \sum_{k=0}^\infty  h_{kl}h_{k+j-i,l+s}\sum_{t=1}^n  Z_{i-k,t-l}^2 \Big)_{ij}\\
&= \Big( \sum_{l_1,l_2=0}^\infty
\sum_{k_1,k_2=0}^\infty h_{k_1,l_1}h_{k_2,l_2}\sum_{t=1}^n
Z_{i-k_1,t-l_1} Z_{j-k_2,t+s-l_2} \1_{\{ l_2=l_1+s, i-k_1=j-k_2  \}^c}\Big)_{ij}\,.
\end{split}
\end{equation*}
Here the indicator refers to the index set for which  $l_2=l_1+s, i-k_1=j-k_2$ does not hold.
%Therefore $\X(0) \X(s)' - \bfQr(s) = \bfR(s)$. 
For $l_1,l_2,k_1,k_2\ge 0$, we define the $p\times p$ matrices $\bfN_{l_1,l_2,k_1,k_2}(s)$,
\begin{equation*}
\big(\bfN_{l_1,l_2,k_1,k_2}(s) \big)_{i,i-k_1+k_2} = \1_{\{l_2=l_1+s\}} \sum_{t=1}^n  Z_{i-k_1,t-l_1}^2\,, \quad i= 1+(-k_1+k_2)_- , \ldots, p-(-k_1+k_2)_+\,,
\end{equation*}
all other entries being zero,
and the $p\times n$ matrices
\beao
\z(l,k)=\z_n(l,k)=(Z_{i-l,t-k})_{i=1,\ldots,p;t=1,\ldots,n}\,, \quad l,k\in\Z\,.
\eeao
We have for $i,j=1,\ldots,p$,
\beao%\label{eq:Ygsrg}
\Big(\z(k_1,l_1)\z(k_2,l_2-s)'- \bfN_{l_1,l_2,k_1,k_2}(s)\Big)_{ij} = \sum_{t=1}^n Z_{i-k_1,t-l_1} Z_{j-k_2,t+s-l_2}\1_{\{ l_2=l_1+s, i-k_1=j-k_2  \}^c}\,
\eeao
and therefore
\begin{equation}\label{eq:ksdfoks1}
\bfC_n(s) - \bfQr_n(s)=\sum_{l_1,l_2, k_1,k_2=0}^\infty h_{k_1,l_1}h_{k_2,l_2} \Big(\z(k_1,l_1)\z(k_2,l_2-s)'- \bfN_{l_1,l_2,k_1,k_2}(s)\Big)\,.
\end{equation}

Using the techniques from the proof of Theorem 5.1 in \cite{heiny:mikosch:2017:iid}, one obtains
\begin{equation}\label{eq:ksdfoks2}
a_{np}^{-2} \Big\|\z(k_1,l_1)\z(k_2,l_2-s)'- \bfN_{l_1,l_2,k_1,k_2}(s)\Big\|_2 \cip 0\,, \quad \nto\,, k_i,l_i\ge 0\,.
\end{equation}
In view of \eqref{eq:ksdfoks1},\eqref{eq:ksdfoks2} and the fact that 
\begin{equation*}
\sum_{l_1,l_2, k_1,k_2=0}^\infty |h_{k_1,l_1}h_{k_2,l_2}|< \infty\,,
\end{equation*}
the proof is complete.

\end{proof}

\subsection*{Truncation of $\bfQr_n(s)$} In this step we show that it suffices to truncate the infinite series of 
the entries of ${\bfQr}_n(s)$. 
For $m\ge 1$, define
\begin{equation*}
\bfQrm_n(s)_{ij}=\sum_{l=0}^m
\sum_{k=0}^m h_{k,l}h_{k+j-i,l+s}\sum_{t=1}^n \Zt_{i-k,t-l}\,.
\end{equation*}

\begin{lemma} \label{lem:Qm}
Assume the conditions of Theorem~\ref{thm:mainn} and $s\ge 0$.
Then
\begin{equation*}
\lim_{m\to\infty} \limsup_{\nto} \P\big(a_{np}^{-2}\big\| \bfQr_n(s)- \bfQrm_n(s) \big\|_2>\vep\big)=0\,,\quad \vep>0\,.
\end{equation*}
\end{lemma}

\begin{proof}
We observe that
\beao
\bfQr_n(s)_{ij}-\bfQrm_n(s)_{ij}=\sum_{k\vee l>m} h_{k,l}h_{k+j-i,l+s}\sum_{t=1}^n \Zt_{i-k,t-l}\,.
\eeao
Now one can follow the proof of Lemma 5.1 in \cite{davis:mikosch:pfaffel:2016} with particular focus on 
$I_n^{(1)}$. Notice that the only difference is the appearance of the additional quantity $s$ in $h_{k+j-i,l+s}$.
\end{proof}

%-----------------------------------------------------------------------
\subsection*{Truncation of the matrix $\bfM(s)$}
Recall the definition of $\bfM_a^{(m)}(s)$ in \eqref{eq:matrixmtr}. 
\begin{lemma}\label{lem:1}
Assume the conditions of Theorem~\ref{thm:mainn} and $s\ge 0$.
Then
\begin{equation*}
a_{np}^{-2}\Big\| \bfQrm_n(s) -\sum_{a=1}^p
  \Dr_a \bfM_a^{(m)}(s)\Big\|_2 \cip 0, \quad \nto.
\end{equation*}
\end{lemma}

\begin{proof}
We start by assuming $\E[Z^2]=\infty$. We have
\beao
\lefteqn{a_{np}^{-2} \big\| \bfQrm(s) -\sum_{a=1}^p
  \Dr_a \bfM_a^{(m)}(s) \big\|_2}\nonumber\\ &\le& a_{np}^{-2}\big\|\bfQrm(s) -\sum_{a=-m}^p
  \Dr_a \bfM_a^{(m)}(s)\big\|_2+a_{np}^{-2} \big\|\sum_{a=-m}^0 \Dr_a \bfM_a^{(m)}(s)\big\|_2\,.
\eeao
We will show that
\begin{eqnarray}\label{eq:help2}
a_{np}^{-2} \big\|\sum_{a=-m}^0 \Dr_a \bfM_a^{(m)}(s)\big\|_2\le a_{np}^{-2}\max_{i=-m,\ldots,0} |\Dr_i|\;
\sum_{a=-m}^0\|\bfM_a^{(m)}(s)\|_2\cip 0\,.
\end{eqnarray}\noindent 
Indeed, $\|\bfM_a^{(m)}(s)\|_2< \infty$ for each $a$. Moreover, for every fixed $i$ and $\alpha\in (0,4)\backslash\{2\}$,
$(a_{n}^{-2} \Dr_i)$ converges in \ds\ to an $\alpha/2$-stable \rv\
as $\nto$. In the case $\alpha=2$, we still have
\beao
a_n^{-2} \big(D_i-n\,\E[Z^2 \1(|Z|\le a_n)]\big)\std \xi_1
\eeao
for a 1-stable \rv . In this case, we also have 
\beao
\dfrac{n}{a_{np}^2} \E[Z^2 \1(|Z|\le a_n)]= \big[\dfrac{n}{a_n^2} \E [Z^2 \1(|Z|\le a_n)]\big]\, 
\dfrac{a_n^2}{a_{np}^2}\,,
\eeao
where the first term in brackets is a \slvary\ \fct\ of $n$ by virtue of Karamata's theorem, 
while the second one converges to zero
at the rate of some positive power of $n$ provided $\beta>0$, hence the \rhs\ converges to zero in the latter case.  Fortunately, the case $\alpha=2$, $\beta=0$
is excluded by the assumptions of  Theorem~\ref{thm:mainn}.
Combining all the facts from above, \eqref{eq:help2} follows.
\par
Observe that
\begin{equation}\label{eq:ghjgh}
 \Big( \bfQrm(s) -\sum_{a=-m}^p
  \Dr_a \bfM_a^{(m)}(s)\Big)_{ij}
    = \sum_{t=1}^n \sum_{k=0}^m \sum_{l=0}^m h_{kl}h_{j-i+k, l+s}
    \Zt_{i-k, t-l} - \sum_{k=i-p}^{i+m} \Dr_{i-k} (\bfM_{i-k}^{(m)}(s))_{ij}
\end{equation}
  and note that $(\bfM_{i-k}^{(m)}(s))_{ij}$ is non-zero only if $i-k \leq i \leq
  i-k+m$, i.e., $0 \leq k \leq m$. This fact and the structure of $\bfM_a^{(m)}$ imply that the right-hand side of \eqref{eq:ghjgh} can be written in the following form:
\begin{equation}\label{eq:lgfh}
\sum_{k=0}^m \sum_{l=0}^m h_{kl} h_{j-i+k, l+s}\left(
      \sum_{t=1}^l \Zt_{i-k, t-l} - \sum_{t=n-l+1}^n \Zt_{i-k, t}
    \right) = \bfI_{ij}^{(1)}-\bfI_{ij}^{(2)}.
\end{equation}
For $a_{np}^{-2}\big\|\bfQrm(s) -\sum_{a=-m}^p
  \Dr_a \bfM_a^{(m)}(s)\big\|_2 \cip 0$ it suffices to show that
$a_{np}^{-2} \|\bfI^{(i)}\|_2\cip 0$, $i=1,2.$
We will show the limit relation for $i=1$; the case $i=2$ is
analogous. 
In the sequel, we interpret $h_{kl}$ as $0$ for $\max(k,l)>m$ or $\min(k,l)<0$. 
For the non-symmetric $\bfI^{(1)}$, we have 
%$\twonorm{\bfI^{(1)}}\le\|{\bfI^{(1)}}\|_1$ where
\begin{equation*}
\|{\bfI^{(1)}}\|_2 \le \sum_{i=1}^p \sum_{j=1}^p|\bfI^{(1)}_{ij}|\,.
\end{equation*}
Since $(h_{kl})$ contains only finitely many non-zero elements it is not difficult to see that it suffices to prove
\beam\label{eq:joh6}
a_{np}^{-2}\sum_{k,l=0}^m  \sum_{t=1}^m \sum_{i=1}^p \Zt_{i-k, t-l}\stp 0\,.
\eeam
 Fix $k,t-l$. If $\alpha\in (0,4)\backslash\{2\}$ or $\alpha=2$ and $\E[Z^2]<\infty$ then  $(a_p^{-2}  \sum_{i=1}^p \Zt_{i-k, t-l})$ has an $\alpha/2$-stable limit.  
Therefore \eqref{eq:joh6} holds. 
Next consider the case $\alpha=2$ and $\E[Z^2]=\infty$. Write $c_p=\E [Z^2\1(|Z|\le a_p)]$ and observe that $(c_p)$ is a \slvary\ \seq .  
In this case, 
\beao
a_p^{-2}\sum_{i=1}^p \Zt_{i-k, t-l}=a_p^{-2}\sum_{i=1}^p (\Zt_{i-k, t-l}-c_p)+ \dfrac{p\,c_p}{a_p^2}\,.
\eeao
The first quantity converges to a totally skewed to the right 1-stable \ds , while $(p\,c_p/a_p^2)$ is a \slvary\ \seq .
Since 
\beao
\dfrac{p\,c_p}{a_{np}^2} =\dfrac{p\,c_p}{a_p^2}\,\dfrac{a_{p}^2}{a_{np}^2}\to 0\,,
\eeao
we may conclude that \eqref{eq:joh6} also holds in this case. This finishes the proof for $\E[Z^2]=\infty$. The case $\E[Z^2]<\infty$ is analogous.
\end{proof}
\begin{remark}\label{re:cen}{\em
Note that the stable convergence which we used to justify \eqref{eq:help2} requires centering in the case $\E[Z^2]<\infty$. From \eqref{eq:Ztilde} we see that one only centers if $\alpha>2(1+\beta)$. 
Fortunately, if $\alpha\in [2, 2(1+\beta))$ the centering is negligible in view of 
$\E[Z^2]\,n/a_{np}^2 \to 0$.
%\begin{equation}\label{eq:addcentering}
If $\alpha>2(1+\beta)$, we have $n/a_{np}^2 \to \infty$. This explains the appearance of the critical value $\alpha^{\star}=2(1+\beta)$ in many places within this paper; see also \cite{heiny:mikosch:2017:iid}. For $\alpha=\alpha^{\star}$, the asymptotic behavior of $n/a_{np}^2$ depends on the slowly varying function $L$ in the distribution of $Z$, which was defined in \eqref{eq:27}.
}\end{remark}

%-----------------------------------------------------------------------
\subsection*{Truncation of the sum}

From \eqref{eq:help6} recall  the definition of the order statistics
\begin{equation*}
\Dr_{(p)}^2=\Dr_{L_p}^2<\cdots  <  \Dr_{(1)}^2=\Dr_{L_1}^2\quad \as
\end{equation*}\noindent
of the iid sequence $\Dr_1^2,\ldots,\Dr_p^2$. % defined in \eqref{eq:ll}.
Here we assume without loss of generality that there are no ties in the
sample.
Otherwise, if two or more of the $\Dr_i^2$'s are equal, randomize the
corresponding $L_i$'s over the respective indices.

We choose an integer sequence $k=k_p\to\infty$ such that $k_p^2=o(p)$ as $\nto$
and recall the definition of the event $A_n$ from \eqref{eq:setA}.
Since the $\Dr_i$'s are iid, $L_1,\ldots,L_k$ have a uniform distribution
on the set of
distinct $k$-tuples from $(1,\ldots,p)$ and
\begin{equation}\label{eq:31}
\lim_{\nto} \P(A_n^c)\le \lim_{\nto} k(k-1) \dfrac{pm (p-2)\ldots (p-k+1)}{p(p-1)\ldots (p-k+1)}\le \lim_{\nto}
\dfrac{k^2 m}{p-1}= 0 \,.
\end{equation}\noindent
In this step of the proof we approximate $\sum_{i=1}^p \Dr_i \bfM_i^{(m)}(s)$
by the matrix $\sum_{i=1}^k \Dr_{L_i}  \bfM_{L_i}(s)^{(m)}(s)$ which is block
diagonal with high probability.

\begin{lemma}\label{lem:2}
Assume the conditions of Theorem~\ref{thm:mainn} and $s\ge 0$. Consider an integer sequence  $(k_p)$ such that
$k_p\to\infty$ and $k_p^2=o(p)$ as $\nto$. Then
\begin{equation*}
a_{np}^{-2}\Big\|\sum_{i=1}^p
  \Dr_i \bfM_i^{(m)}(s) -\sum_{i=1}^k
  \Dr_{L_i} \bfM_{L_i}^{(m)}(s)\Big\|_2\cip 0\,,
\quad \nto \,.
\end{equation*}
\end{lemma}

\begin{proof}
%\textbf{The case $\E[Z^2]=\infty$ or $\alpha>2(1+\beta)$.}  
We have
\begin{equation*}
a_{np}^{-2}\Big\|\sum_{i=1}^p
  \Dr_i \bfM_i^{(m)}(s) -\sum_{i=1}^k
  \Dr_{L_i} \bfM_{L_i}^{(m)}(s)\Big\|_2
=a_{np}^{-2}\Big\| \sum_{i=k+1}^p
  \Dr_{L_i} \bfM_{L_i}^{(m)}(s)\Big\|_2\,,
\end{equation*}\noindent
and therefore it suffices to show that the \rhs\ converges to zero in
probability. Since the $\bfM_{i}^{(m)}(s), i=1,\ldots,p$, consist of block matrices of size $m$ shifted by $i$, at most $2m$ of them can overlap. By Cauchy's interlacing theorem, see \cite[Lemma~22]{tao:vu:2012}, we obtain for $\delta>0$,
\begin{equation}\label{eq:98}
\P\Big(a_{np}^{-2}\Big\|\sum_{i=k+1}^p \Dr_{L_i}
\bfM_{L_i}^{(m)}(s)\Big\|_2>\delta\Big)
\le \P\Big(c\, a_{np}^{-2} m |\Dr_{L_{k+1}}|>\delta\Big) +\P(A_n^c) \to 0\,, \quad \nto\,.
\end{equation}
%In the case  $\E[Z^2]<\infty$ and $\alpha\in [2, 2(1+\beta))$ one may, for example, assume centering in view of \eqref{eq:addcentering}.
\begin{comment}
\textbf{The case  $\E[Z^2]<\infty$ and $\alpha\in [2, 2(1+\beta))$.} It is sufficient to prove that
\begin{equation*}
\begin{split}
a_{np}^{-2}\Big\| \sum_{i=k+1}^p \Dr_{L_i} \bfM_{L_i}^{(m)}(s)\Big\|_2
&\le a_{np}^{-2}\Big\| \sum_{i=k+1}^p (\Dr_{L_i}-\E[D]) \bfM_{L_i}^{(m)}(s)\Big\|_2\\
 &\quad + \E[\Dr]\Big\| \sum_{i=k+1}^p \bfM_{L_i}^{(m)}(s)  \Big\|_2 \cip 0\,, \quad \nto\,.
\end{split}
\end{equation*}
Using the same argument as above, we conclude that for $\delta>0$, 
\begin{equation*}
\P\Big(a_{np}^{-2}\Big\|\sum_{i=k+1}^p (\Dr_{L_i}-\E[D])
\bfM_{L_i}^{(m)}(s)\Big\|_2>\delta\Big)
\le \P\Big(c\, a_{np}^{-2} m \Dr_{(k+1)}>\delta\Big)\to 0\,, \quad \nto\,.
\end{equation*}

Since the $\bfM_{i}^{(m)}(s), i=1,\ldots,p$, consist of block matrices of size $m$ shifted by $i$, at most $2m$ of them can overlap. By Cauchy's interlacing theorem, see \cite[Lemma~22]{tao:vu:2012},
\begin{equation*}
\frac{\E[\Dr]}{a_{np}^2} \Big\| \sum_{i=k+1}^p \bfM_{L_i}^{(m)}(s)  \Big\|_2 \le c m \frac{\E[\Dr]}{a_{np}^2} \to 0, \quad \nto.
\end{equation*}
Now we can follow the lines of the proof in the case $\E[Z^2]<\infty$ to show the desired result.
\end{comment}
\end{proof}

\subsection*{Conclusion} We found several approximations of $\bfC_n(s)$. The proof of Theorem~\ref{prop:main} consists of a direct application of Proposition~\ref{prop:withoutQ} and Lemmas~\ref{lem:Qm}-\ref{lem:2}. 

\begin{remark}[The case $\alpha=2(1+\beta)$.]\label{rem:speciala}
{\em For clarity of presentation we excluded this case in \eqref{eq:sample}. If $\alpha= 2(1+\beta)$, the definition of $\bfC_n(s)$ in \eqref{eq:sample} depends on the distribution of $Z$ and the growth of $p$. More precisely, 
if $ n/a_{np}^2 \to 0$ or $\E[Z^2]=\infty$, we set $\bfC_n(s)=\X_n(0)\X_n(s)'$. Otherwise we define $\bfC_n(s)=\X_n(0)\X_n(s)'-\E[\X_n(0)\X_n(s)']$. The proofs are exactly the same, except in the case $\alpha=2$ and $\beta=0$ where one has to additionally distinguish between finite or infinite variance of $Z$. }
\end{remark}

%------------------------------------------------
\appendix
\section{Perturbation theory for eigenvectors}\label{appendix:A}\setcounter{equation}{0}

We state Proposition~A.1 in Benaych-Georges and P\'{e}ch\'{e} \cite{benaych:peche:2014}.
\begin{proposition}\label{prop:perturbation}
Let $\bfH$ be a Hermitean matrix and $\bfv$ a unit vector such that for some  $\la\in \R$, $\vep>0$,
\begin{equation*}
\bfH\,\bfv= \la\,\bfv + \vep\, \bfw\,,
\end{equation*}
where $\bfw$ is a unit vector such that $\bfw \perp \bfv$.
\begin{enumerate}
\item
Then $\bfH$ has an eigenvalue $\lambda_{\vep}$ \st\ $|\la-\la_\vep|\le \vep$.
\item
If $\bfH$ has only one eigenvalue $\la_\vep$ (counted with multiplicity) \st\ $|\la-\la_\vep|\le \vep$ and all other eigenvalues 
are at distance at least $d>\vep$ from $\la$. Then for a unit eigenvector $\bfv_\vep$ associated with $\lambda_{\vep}$ we have
\begin{equation*}
\ltwonorm{\bfv_{\vep} - \Proj_\bfv(\bfv_{\vep})} \le \frac{2\, \vep}{d-\vep}\,,
\end{equation*}
where $\Proj_{\bfv}$ denotes the orthogonal projection onto \rm{Span}$(\bfv)$.
\end{enumerate}
\end{proposition}

\section*{Acknowledgments}\setcounter{equation}{0}

We thank Olivier Wintenberger for reading the 
manuscript and fruitful discussions. This research was started 
when both authors visited the Department of Statistics at
Columbia University. We are most grateful to Richard A. Davis for
his hospitality and stimulating discussions.

\bibliography{libraryjohannes}

\end{document}